\documentclass[11pt]{amsart}
\usepackage[margin=1in]{geometry}

\usepackage{amssymb}
\usepackage{amsthm}
\usepackage{amsmath}
\usepackage{mathrsfs}
\usepackage{amsbsy}
\usepackage{bm}
\usepackage{hyperref}
\usepackage{tikz}
\usepackage{array}
\usepackage{enumerate}
\usepackage{xcolor}
\usepackage{bbm}
\usepackage{comment}
\usepackage{mathtools}

\definecolor{LightBlue}{rgb}{0.392,0.392,1} 
\definecolor{Green}{rgb}{0,0.922,0}
\definecolor{DarkGreen}{rgb}{0,0.5,0}
\definecolor{MildGreen}{rgb}{0,0.784,0}
\definecolor{NormalGreen}{rgb}{0,0.8,0}
\definecolor{LightGreen}{rgb}{0,0.922,0}
\definecolor{Magenta}{rgb}{1,0,0.6}
\definecolor{Yellow}{rgb}{0.95,0.95,0}
\definecolor{lavender}{rgb}{0.4,0,1}

\usepackage[noabbrev,capitalise,nameinlink]{cleveref}

\hypersetup{colorlinks=true, citecolor=lavender, linkcolor=lavender,urlcolor=lavender}

\usepackage{hhline}
\allowdisplaybreaks
\usepackage[noadjust]{cite}

\usepackage{caption}
\usepackage[noabbrev,capitalise,nameinlink]{cleveref}
\crefname{conjecture}{Conjecture}{Conjectures}

\newtheorem{theorem}{Theorem}[section]

\newtheorem{lemma}[theorem]{Lemma}

\theoremstyle{definition}

\newtheorem{remark}[theorem]{Remark}
\newtheorem{example}[theorem]{Example}

\usepackage{etoolbox}

\newcommand{\includeSymbol}[1]{\ensuremath{%
	\mathchoice
		{\raisebox{-.4mm}{\includegraphics[height=2.1ex]{#1}}}	
		{\raisebox{-.4mm}{\includegraphics[height=2.1ex]{#1}}}
		{\raisebox{-.3mm}{\includegraphics[height=1.6ex]{#1}}}
		{\raisebox{-.2mm}{\includegraphics[height=1ex]{#1}}}
}}

\robustify{\includeSymbol}
\newcommand{\Elbows}{\includeSymbol{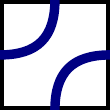}}

\robustify{\includeSymbol}
\newcommand{\Cross}{\includeSymbol{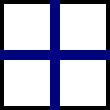}}

\robustify{\includeSymbol}
\newcommand{\SEElbow}{\includeSymbol{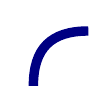}}

\robustify{\includeSymbol}
\newcommand{\NWElbow}{\includeSymbol{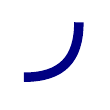}}

\newcommand{\bT}{\mathscr{T}}
\newcommand{\qq}{\mathfrak{q}}
\newcommand{\LL}{\mathfrak{L}}
\newcommand{\NN}{\mathfrak{N}}
\newcommand{\MM}{\mathfrak{M}}
\newcommand{\TT}{\mathrm{T}}
\newcommand{\PD}{\mathrm{PD}}
\newcommand{\XX}{\mathrm{X}}
\newcommand{\ff}{\mathrm{F}}
\newcommand{\pp}{\mathfrak{p}}

\newcommand{\N}{\mathrm{N}}
\newcommand{\E}{\mathrm{E}}
\newcommand{\sub}{\mathsf{sub}}
\newcommand{\inv}{\mathrm{inv}}
\newcommand{\ww}{\mathsf{w}}
\newcommand{\bS}{\mathscr{S}}
\newcommand{\SSS}{\mathfrak{S}}
\newcommand{\fb}{\mathfrak{b}}
\newcommand{\bb}{\mathsf{b}}
\newcommand{\hh}{\mathbf{h}}
\newcommand{\dd}{\mathbf{d}}
\newcommand{\ZZ}{\mathbb{Z}}
\newcommand{\occ}{\mathfrak{O}}

\newcommand{\dfn}[1]{\textcolor{blue}{\emph{#1}}}

\begin{document}

\title[]{Random Subwords and Pipe Dreams}
\subjclass[2010]{}

\author[]{Colin Defant}
\address[]{Department of Mathematics, Harvard University, Cambridge, MA 02138, USA}
\email{colindefant@gmail.com}

\begin{abstract}
Fix a probability $p\in(0,1)$. Let $s_i$ denote the transposition in the symmetric group $\mathfrak{S}_n$ that swaps $i$ and $i+1$. Given a word $\mathsf{w}$ over the alphabet $\{s_1,\ldots,s_{n-1}\}$, we can generate a random subword by independently deleting each letter of $\mathsf{w}$ with probability $1-p$. For a large class of starting words $\mathsf{w}$---including all alternating reduced words for the decreasing permutation---we compute precise asymptotics (as $n\to\infty$) for the expected number of inversions of the permutation represented by the random subword. This result can also be seen as an asymptotic formula for the expected number of inversions of a permutation represented by a certain random (non-reduced) pipe dream. In the special case when $\mathsf{w}$ is the word $(s_{n-1})(s_{n-2}s_{n-1})\cdots(s_1s_2\cdots s_{n-1})$, we find that the expected number of inversions of the permutation represented by the random subword is asymptotically equal to
\[\frac{2\sqrt{2}}{3\sqrt{\pi}}\sqrt{\frac{p}{1-p}}\,n^{3/2};\] this settles a conjecture of Morales, Panova, Petrov, and Yeliussizov. 
\end{abstract} 

\maketitle

\section{Introduction}\label{sec:intro}

\subsection{Random Subwords} 
Let $\SSS_n$ denote the symmetric group whose elements are permutations of the set $[n]\coloneq\{1,\ldots,n\}$. For $1\leq i\leq n-1$, let $s_i$ be the adjacent transposition in $\SSS_n$ that swaps $i$ and $i+1$. When $n$ is understood, we refer to a finite word over the alphabet $\{s_1,\ldots,s_{n-1}\}$ simply as a \dfn{word}. Let $\occ_i(\ww)$ denote the number of occurrences of $s_i$ in the word $\ww$. We say $\mathsf{w}$ is \dfn{alternating} if for all $i\in[n-2]$, the letters $s_i$ and $s_{i+1}$ alternate within $\mathsf{w}$. An \dfn{inversion} of a permutation $w\in\SSS_n$ is a pair $(j,j')$ of integers such that $1\leq j<j'\leq n$ and $w^{-1}(j)>w^{-1}(j')$. The minimum length of a word representing $w$ is equal to the number of inversions of $w$, which we denote by $\inv(w)$. A \dfn{reduced word} for $w$ is a word of length $\inv(w)$ that represents $w$. 

Throughout this article, let us fix a probability $p\in(0,1)$ and a sufficiently small real number $\varepsilon>0$. Given a word $\ww$, let $\sub_p(\ww)$ be the random subword of $\ww$ obtained by independently deleting each letter of $\ww$ with probability $1-p$. Thus, if $\ww$ has length $K$, then $\sub_p(\ww)$ has expected length $pK$. The word $\sub_p(\ww)$ is likely not reduced, so it is natural to ask how much we can reduce it. In other words, we are interested in the expected number of inversions of the permutation represented by $\sub_p(\ww)$. Our main result gives precise asymptotics for this expected value (as $n\to\infty$) for a large class of alternating words. In what follows, we write $f(n)\ll g(n)$ (or $g(n)\gg f(n)$) to mean that $f(n)=O(g(n))$. For $\alpha,\beta>0$, we let $\log^\alpha\beta=(\log\beta)^{\alpha}$. 

\begin{theorem}\label{thm:main}
For each $n\geq 2$, let $\ww^{(n)}$ be an alternating word over the alphabet $\{s_1,\ldots,s_{n-1}\}$. Assume that \[\max_{i\in[n-1]}\occ_i(\ww^{(n)})=o\left( n^2/\log^{1+\varepsilon} n\right)\quad\text{and}\quad |\{i\in[n-1]:\occ_i(\ww^{(n)})\leq n^{2\varepsilon}\}|=o(n^{1-\varepsilon}).\]
Let $v^{(n)}$ be the random permutation represented by $\sub_p(\ww^{(n)})$. Then 
\[\mathbb E(\inv(v^{(n)}))=\left(\sqrt{2/\pi}+o(1)\right)\sqrt{\frac{p}{1-p}}\,\sum_{i=1}^{n-1}\sqrt{\occ_i(\ww^{(n)})}.\]
\end{theorem}

\begin{example}\label{exam:staircase1}
Let $\ww^{(n)}$ be the \dfn{staircase reduced word} $(s_{n-1})(s_{n-2}s_{n-1})\cdots(s_1s_2\cdots s_{n-1})$ (the parentheses are just for clarity). Then $\ww^{(n)}$ is a reduced word for the decreasing permutation in $\SSS_n$. In this special case, the random permutation $v^{(n)}$ represented by $\sub_p(\ww^{(n)})$ was studied by Morales, Panova, Petrov, and Yeliussizov \cite{MPPY}, who obtained the estimates \[\frac{2}{3\sqrt{\pi}}\sqrt{\frac{p}{1-p}}\,(1+o(1))n^{3/2}\leq\mathbb E(\inv(v^{(n)}))\leq\frac{4}{3\sqrt{\pi}}\sqrt{\frac{p}{1-p}}\,(1+o(1))n^{3/2}.\] They also conjectured that there exists a constant $\varkappa\approx 0.53$ such that \[\mathbb E(\inv(v^{(n)}))=\varkappa\,\sqrt{\frac{p}{1-p}}\,(1+o(1))n^{3/2}.\] Since $\occ_i(\ww^{(n)})=i$, it follows from \cref{thm:main} that 
\begin{align*}\mathbb E(\inv(v^{(n)}))&=\left(\sqrt{2/\pi}+o(1)\right)\sqrt{\frac{p}{1-p}}\,\sum_{i=1}^{n-1}\sqrt{i} \\ 
&=\frac{2\sqrt{2}}{3\sqrt{\pi}}\sqrt{\frac{p}{1-p}}\,(1+o(1))n^{3/2}.
\end{align*} 
This proves the aforementioned conjecture and shows that the exact value of $\varkappa$ is $\frac{2\sqrt{2}}{3\sqrt{\pi}}$. 
\end{example} 

\begin{example}\label{exam:bipartite1}
Consider the words $\mathsf{x}_n=s_1s_3\cdots s_{2\left\lfloor n/2\right\rfloor-1}$ (consisting of the odd-indexed adjacent transpositions) and $\mathsf{y}_n=s_2s_4\cdots s_{2\left\lfloor (n-1)/2\right\rfloor}$ (consisting of the even-indexed adjacent transpositions). Let 
\[\ww^{(n)}=\begin{cases} (\mathsf{x}_n\mathsf{y}_n)^{n/2} & \mbox{if } n\text{ is even}; \\ (\mathsf{x}_n\mathsf{y}_n)^{(n-1)/2}\mathsf{x}_n & \mbox{if } n\text{ is odd.} \end{cases}\] Then $\ww^{(n)}$ is a reduced word for the decreasing permutation in $\SSS_n$. As before, let $v^{(n)}$ be the random permutation represented by $\sub_p(\ww^{(n)})$. For each $i$, we have $|\occ_i(\ww^{(n)})-n/2|\leq\frac{1}{2}$, so it follows from \cref{thm:main} that 
\[\mathbb E(\inv(v^{(n)}))=\frac{1}{\sqrt{\pi}}\sqrt{\frac{p}{1-p}}\,(1+o(1))n^{3/2}.
\]
\end{example} 

\begin{example}\label{exam:bipartite2}
Let $(\varrho_n)_{n\geq 2}$ be a sequence of integers such that $n^{2\varepsilon}<\varrho_n=o(n^2/\log^{1+\varepsilon} n)$. Let $\mathsf{x}_n$ and $\mathsf{y}_n$ be as in \cref{exam:bipartite1}. Let $x_n$ and $y_n$ be the random permutations represented by $\sub_p(\mathsf{x}_n)$ and $\sub_p(\mathsf{y}_n)$, respectively. Let $v^{(n)}=x_n^{(1)}y_n^{(1)}x_n^{(2)}y_n^{(2)}\cdots x_n^{(\varrho_n)}y_n^{(\varrho_n)}$, where $x_n^{(1)},\ldots,x_n^{(\varrho_n)}$ are independent copies of $x_n$ and $y_n^{(1)},\ldots,y_n^{(\varrho_n)}$ are independent copies of $y_n$. Then $v^{(n)}$ has the same distribution as the permutation represented by $\sub_p((\mathsf{x}_n\mathsf{y}_n)^{\varrho_n})$, so \cref{thm:main} tells us that \[\mathbb E(\inv(v^{(n)}))=\sqrt{2/\pi}\sqrt{\frac{p}{1-p}}\,(1+o(1))n\sqrt{\varrho_n}.
\]  
\end{example}

\begin{remark}
It is likely that some of the hypotheses on the sequence $(\ww^{(n)})_{n\geq 2}$ in \cref{thm:main} could be weakened. We have chosen the hypotheses so as to simplify our arguments while still remaining fairly general. In particular, if each $\ww^{(n)}$ is an alternating reduced word for the decreasing permutation in $\SSS_n$, then these hypotheses are satisfied (see \cite{GGJNP}). 
\end{remark}

\subsection{Random Pipe Dreams} 

We can apply a \dfn{commutation move} to a word by swapping two adjacent letters that commute with each other. Two words are \dfn{commutation equivalent} if one can be obtained from the other by applying a sequence of commutation moves. If $\ww$ and $\ww'$ are commutation equivalent, then the random permutations represented by $\sub_p(\ww)$ and $\sub_p(\ww')$ have the same distribution. Thus, we typically only care about words up to commutation equivalence.

For $x,y\in\ZZ$, define the \dfn{box} $\fb(x,y)$ to be the axis-parallel unit square in $\mathbb R^2$ whose lower left vertex is $(x,y)$. Let $\dd(\fb(x,y))=x-y$. A set of boxes is called a \dfn{shape}; we view a shape $\bS$ as a poset in which $\fb(x,y)\leq\fb(x',y')$ if and only if $x\leq x'$ and $y\leq y'$. A \dfn{linear extension} of $\bS$ is an ordering of the boxes in $\bS$ such that for all $\bb,\bb'\in\bS$ satisfying $\bb\leq\bb'$, the box $\bb$ precedes $\bb'$ in the linear extension. We say $\bS$ has \dfn{rank} $r$ if $1\leq\dd(\bb)\leq r$ for all $\bb\in\bS$.   

Suppose $\bS$ is a finite shape of rank $n-1$. We associate to a linear extension $\lambda=(\bb_1,\ldots,\bb_K)$ of $\bS$ the word $\ww_\lambda=s_{\dd(\bb_1)}\cdots s_{\dd(\bb_K)}$ over the alphabet $\{s_1,\ldots,s_{n-1}\}$. The set of words associated to linear extensions of $\bS$ is a commutation equivalence class. We say $\bS$ is \dfn{order-convex} if for all $\bb,\bb'\in\bS$ with $\bb\leq\bb'$, every box $\bb''$ satisfying $\bb\leq\bb''\leq\bb'$ and $1\leq\dd(\bb'')\leq n-1$ is in $\bS$. 

Given an alternating word $\mathsf{w}$ over $\{s_1,\ldots,s_{n-1}\}$, we can find an order-convex shape $\bS$ of rank $n-1$ and a linear extension $\lambda=(\bb_1,\ldots,\bb_K)$ of $\bS$ such that $\ww=\ww_\lambda$. We can then represent a subword $\mathsf{v}$ of $\ww$ pictorially as a \dfn{pipe dream} as follows. If $\mathsf{v}$ contains the copy of the letter $s_{\dd(\bb_r)}$ corresponding to the box $\bb_r$, then fill $\bb_r$ with the \dfn{cross tile} \Cross\,; otherwise, fill $\bb_r$ with the \dfn{bump tile} \Elbows\,. In addition, whenever $\bS$ contains two boxes $\fb(x,x-1)$ and $\fb(x+1,x)$, connect them with a single elbow \!\SEElbow\!. 
Similarly, whenever $\bS$ contains two boxes $\fb(x,x-n+1)$ and $\fb(x+1,x-n+2)$, connect them with a single elbow \!\NWElbow\!. The resulting pipe dream will contain $n$ \dfn{pipes}, which travel from the southwest boundary of $\bS$ to the northeast boundary of $\bS$. We label the pipes $1,\ldots,n$ in order of their starting positions along the southwest boundary (ordered from northwest to southeast). Let $v(1),\ldots,v(n)$ be the labels of the pipes listed in order of their ending positions along the northeast boundary (ordered from northwest to southeast). Then $v$ is the permutation represented by the subword $\mathsf{v}$. 

In the literature, pipe dreams are usually defined in the \dfn{staircase shape} \[\bS_{\mathrm{staircase}}^{(n)}=\{\fb(x,y):x-n+1\leq y\leq 0<1\leq x\},\] corresponding to the commutation equivalence class of the staircase reduced word from \cref{exam:staircase1}. In this context, pipe dreams are prominent objects in Schubert calculus because they provide combinatorial formulas for Schubert polynomials and Grothendieck polynomials \cite{BergeronBilley,FominKirillov,
HPW,KnutsonMiller1,KnutsonMiller2}. 

The left side of \cref{fig:staircase1} shows a pipe dream in the staircase shape $\bS_{\mathrm{staircase}}^{(8)}$; the resulting permutation is $42138657$. The right side of \cref{fig:staircase1} shows a pipe dream corresponding to a subword of the reduced word $(s_1s_3s_5s_7s_2s_4s_6)^4=(\mathsf{x}_8\mathsf{y}_8)^4$ from \cref{exam:bipartite1}; the resulting permutation is $13248657$. \cref{fig:staircase2} shows a random pipe dream in $\bS_{\mathrm{staircase}}^{(50)}$, generated with the parameter $p=1/2$; the resulting permutation of size $50$ has $134$ inversions.  

\begin{figure}[] 
  \begin{center}
  \includegraphics[height=5.506cm]{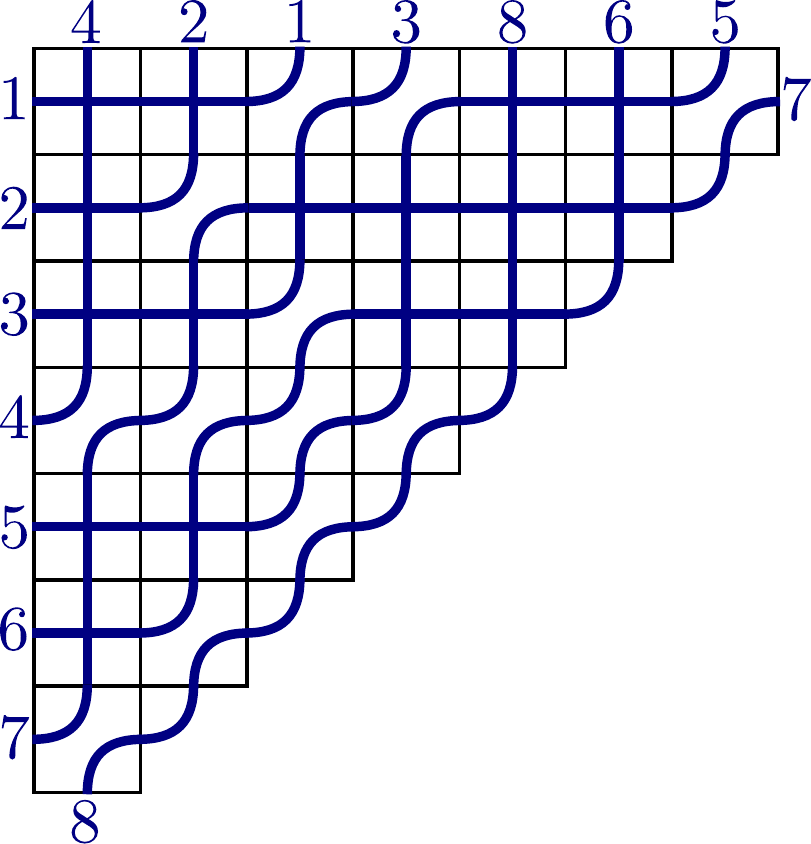}\qquad\qquad\includegraphics[height=5.496cm]{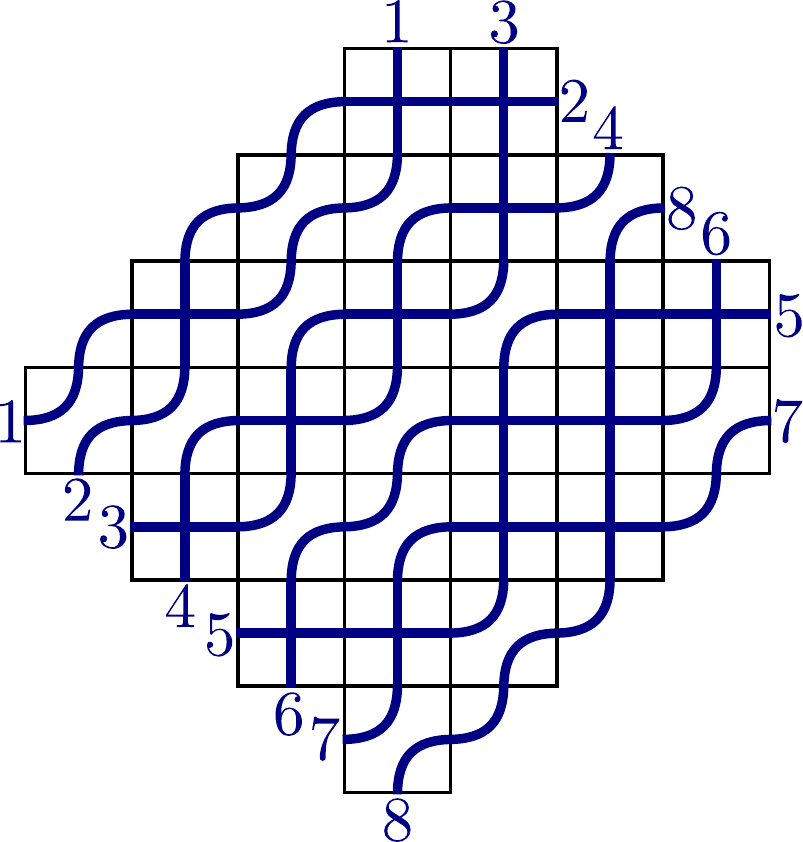}
  \end{center}
\caption{On the left is a pipe dream in the staircase shape $\bS_{\mathrm{staircase}}^{(8)}$; its associated permutation is $42138657$. On the right is a pipe dream representing a subword of $(s_1s_3s_5s_7s_2s_4s_6)^4$; its associated permutation is $13248657$. }\label{fig:staircase1}
\end{figure}

\begin{figure}[]
  \begin{center}
  \includegraphics[width=\linewidth]{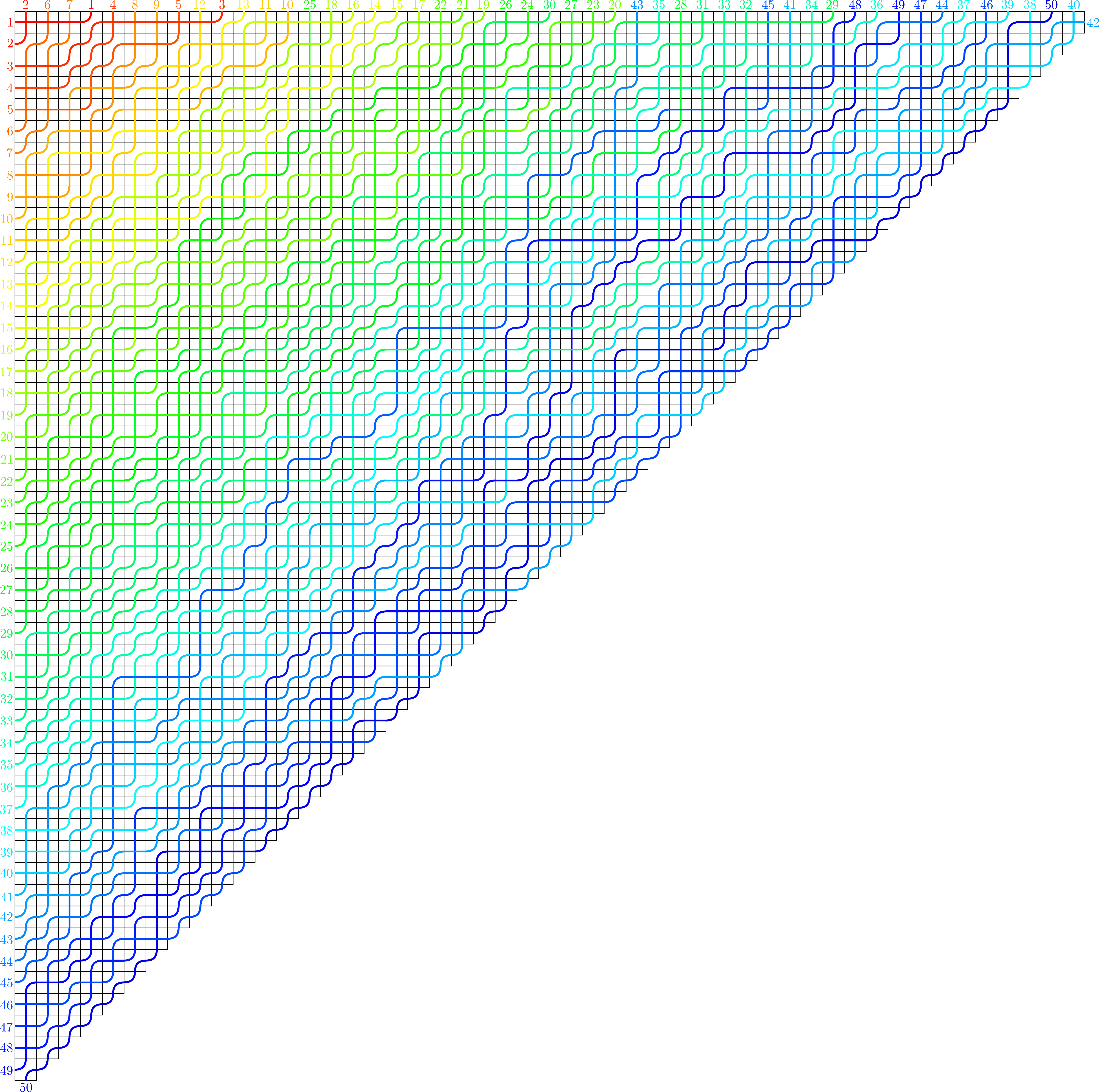}
  \end{center}
\caption{A random pipe dream $\PD_{1/2}(\bS_{\mathrm{staircase}}^{(50)})$ with pipes of different colors. }\label{fig:staircase2}
\end{figure}

We can rephrase \cref{thm:main} in the language of pipe dreams as follows. Given a finite shape $\bS$ of rank $n-1$ and an integer $i\in[n-1]$, let $\occ_i(\bS)$ denote the number of boxes $\bb\in\bS$ such that $\dd(\bb)=i$. Let $\PD_p(\bS)$ be the random pipe dream of shape $\bS$ in which each box is filled with a cross tile with probability $p$ and is filled with a bump tile with probability $1-p$. For each integer $n\geq 2$, let $\bS^{(n)}$ be a finite order-convex shape of rank $n-1$. Assume that \[\max_{i\in[n-1]}\occ_i(\bS^{(n)})=o\left(n^2/\log^{1+\varepsilon} n\right)\quad\text{and}\quad |\{i\in[n-1]:\occ_i(\bS^{(n)})\leq n^{2\varepsilon}\}|=o(n^{1-\varepsilon}).\] 
Let $v^{(n)}$ be the random permutation represented by $\PD_p(\bS^{(n)})$. Then \cref{thm:main} states that 
\[\mathbb E(\inv(v^{(n)}))=\left(\sqrt{2/\pi}+o(1)\right)\sqrt{\frac{p}{1-p}}\,\sum_{i=1}^{n-1}\sqrt{\occ_i(\bS^{(n)})}.\] 
This is the formulation of \cref{thm:main} that we will actually prove.   

There has been substantial interest in subwords of words over $\{s_1,\ldots,s_{n-1}\}$ from topological, geometric, algebraic, and lattice-theoretic perspectives \cite{BCP,BCCP,Cartier,KnutsonMiller1,KnutsonMiller2,
Pilaud1,Pilaud2}. Our probabilistic approach to subwords appears to be new, save for the recent article \cite{MPPY}.\footnote{The current author formulated random subwords and random pipe dreams independently of the authors of \cite{MPPY}, though Great Panova made him aware of their work while it was in preparation after he presented an open problem about random pipe dreams at a conference in honor of Richard Stanley's 80-th birthday.} That said, the last decade has seen other exciting work about random words over $\{s_1,\ldots,s_{n-1}\}$. For example, the articles \cite{Angel,Angel2,Dauvergne,DauvergneVirag} concern random reduced words for the decreasing permutation in $\SSS_n$.

\subsection{Proof Strategy and Outline} 

To prove \cref{thm:main}, we will estimate the probability that two pipes $\pp$ and $\pp'$ form an inversion, and we will then use linearity of expectation. To do this, we imagine that the pipes grow from the southwest boundary to the northeast boundary, and we model the distance between the pipes by a random walk governed by a particular $4$-state Markov chain. In order for this approach to work, we must assume that the starting points of $\pp$ and $\pp'$ lie on a common diagonal line of the form $\{(x,y)\in\mathbb R:x+y=c\}$; roughly speaking, this ensures that the two pipes start growing ``at the same time.'' We first work with \emph{serrated} shapes, which are particular shapes in which all of the starting points of the pipes lie on a single line of this form (see \cref{sec:serrated} for more details). We then show how to reduce the analysis of general shapes in \cref{thm:main} to that of serrated shapes. Thus, our proof crucially requires us to study shapes more general than staircases (as staircases are not serrated). 

There are numerous technical details needed to rigorously carry out the above plan. \cref{sec:concentration} establishes a concentration inequality guaranteeing that each pipe is very unlikely to wander too far from a certain diagonal line. In \cref{sec:distance}, we define and analyze the random walk that governs the distance between two pipes. Much of our analysis is devoted to circumventing the issue that the random walk does not perfectly model the distance between the pipes because the trajectories of the pipes are not completely independent (since the pipes can \emph{kiss}). In \cref{sec:serrated}, we estimate the number of inversions involving a fixed pipe under the assumption that the shape is serrated. \cref{sec:finish_proof} finishes the proof of \cref{thm:main} by reducing the problem from general shapes to serrated shapes. \cref{sec:conclusion} compiles several open problems and suggestions for future research. 

\subsection{Notation and Conventions}\label{sec:notation} 
Whenever we consider a pipe $\pp$ inside a pipe dream in a shape $\bS$, we imagine that it starts at a point on the southwest boundary of $\bS$ and grows north and east over time. At time $0$, the pipe is in a box on the southwest boundary of $\bS$. If $\pp$ is still in the shape at a positive integer time $t$, then at time $t+1$, it either leaves the shape or moves into a new box $\bb_\pp(t+1)$ that is immediately north or immediately east of $\bb_\pp(t)$. To ease notation, let us write 
\[\dd_p(t)=\dd(\bb_p(t)).\] 

Suppose $\pp$ and $\pp'$ are the $j$-th and $j'$-th pipes, respectively, in a pipe dream. We say the pair $(\pp,\pp')$ is an \dfn{inversion} of the pipe dream if $j<j'$ and $\pp$ and $\pp'$ cross an odd number of times. Of course, the number of inversions of a pipe dream is the same as the number of inversions of the permutation represented by the pipe dream.

It is possible for a pipe in a pipe dream to have of length $0$ (i.e., to be a point). In this case, we imagine that the southwest and northeast boundaries of the shape coincide along some path containing the pipe. For instance, the $4$-th and $7$-th pipes in the pipe dream 
\[\begin{array}{l}\includegraphics[height=2.025cm]{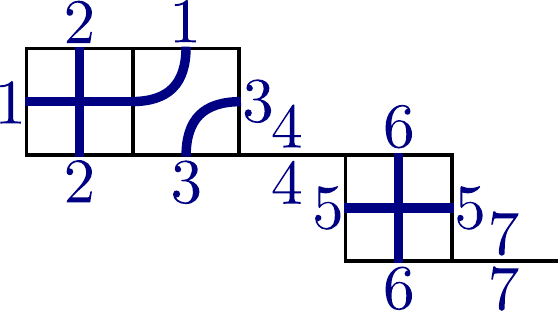}\end{array}\] have length $0$. Whenever we speak of a pipe dream in a shape of rank $n-1$, we implicitly assume (possibly by including pipes of length $0$) that it has exactly $n$ pipes.  

We often omit floor and ceiling symbols; this will not affect the relevant asymptotics. 

\section{Concentration Inequalities for Pipes} \label{sec:concentration} 

Fix an integer $c$. Consider a random pipe dream in the infinite shape 
\begin{equation}\label{eq:RR}
\mathscr{R}=\{\fb(x,y):x,y\in\ZZ,\,x+y\geq c\}
\end{equation} in which each box is filled with a cross tile with probability $p$ and is filled with a bump tile with probability $1-p$. While such a pipe dream does not represent a permutation, it does contain (infinitely many) pipes, each of which is infinitely long, and we can still use the notation from \cref{sec:notation} (except we do not define \emph{inversions} in infinite pipe dreams). We will focus on a single pipe $\pp$. Our goal in this section is to prove \cref{lem:veer}, which is a concentration inequality showing that $\pp$ is very unlikely to veer too far from the diagonal line $\{(x,y)\in\mathbb R^2:x-y=\dd_\pp(0)\}$ in a given amount of time. 

We say $\pp$ \dfn{faces east} (respectively, \dfn{faces north}) at time $t$ if $\pp$ touches the west (respectively, south) side of $\bb_\pp(t)$. Let us say $\pp$ \dfn{turns} at time $t$ if the box $\bb_{\pp}(t-1)$ contains a bump tile. Say $\tau_1,\tau_2,\ldots$ are the times at which $\pp$ turns. For convenience, let $\tau_0=0$. Then for each $k\geq 0$, the quantity $\tau_{k+1}-\tau_k$ is a geometric random variable with parameter $1-p$ (and mean $1/(1-p)$). Moreover, $\dd_\pp(\tau_{k+1})-\dd_\pp(\tau_k)=\pm(\tau_{k+1}-\tau_k-2)$, where the sign depends only on the parity of $k$ and whether $\pp$ faces north or east at time $0$.  

\begin{lemma}\label{lem:half}
If $\pp$ faces east at time $0$, let $\delta=1$; otherwise, let $\delta=-1$. Then for each $k\geq 0$, 
\[\mathbb P(\delta\cdot(\dd_\pp(\tau_k)-\dd_\pp(0))\geq -1)\geq\frac{1}{2}.\] 
\end{lemma}

\begin{proof}
Assume $\pp$ faces east at time $0$ so that $\delta=1$; a completely symmetric argument handles the other case. Fix $k\geq 0$. If $k$ is odd, let $k'=k-1$; otherwise, let $k'=k$. If $k$ is even, then $\dd_\pp({\tau_k})-\dd_\pp(\tau_{k'})=0$. If $k$ is odd, then 
\[\dd_\pp({\tau_{k}})-\dd_\pp({\tau_{k'}})=\tau_k-\tau_{k-1}-2\geq -1.\] In either case, we have $\mathbb P(\dd_\pp({\tau_k})-\dd_\pp(0)\geq -1)\geq\mathbb P(\dd_\pp(\tau_{k'})-\dd_\pp(0)\geq 0)$. Now, 
\begin{align*}\dd_\pp({\tau_{k'}})-\dd_\pp(0)&=\sum_{i=1}^{k'/2}(\dd_\pp({\tau_{2i-1}})-\dd_\pp({\tau_{2i-2}}))+\sum_{i=1}^{k'/2}(\dd_\pp({\tau_{2i}})-\dd_\pp({\tau_{2i-1}})) \\ &=\sum_{i=1}^{k'/2}(\tau_{2i-1}-\tau_{2i-2}-2)-\sum_{i=1}^{k'/2}(\tau_{2i}-\tau_{2i-1}-2).  
\end{align*}
But \[\sum_{i=1}^{k'/2}(\tau_{2i-1}-\tau_{2i-2}-2)\quad\text{and}\quad\sum_{i=1}^{k'/2}(\tau_{2i}-\tau_{2i-1}-2)\] are independent random variables with the same distribution, so $\mathbb P(\dd_\pp(\tau_{k'})-\dd_\pp(0)\geq 0)\geq\frac{1}{2}$. 
\end{proof}

We will need the following concentration inequality for sums of i.i.d.\ geometric random variables. 

\begin{lemma}\label{lem:Chernoff}
Let $G_1,G_2,\ldots$ be independent geometric random variables with parameter $1-p$. For every integer $k\geq 1$ and every $\xi>0$, we have \[\mathbb P(|G_1+\cdots+G_{k}-k/(1-p)|\geq \xi)\leq 2\exp\left(-\frac{1-p}{2}\frac{\xi^2}{k/(1-p)+\xi}\right).\]
\end{lemma}

\begin{proof}
Let 
\[Q_m=\begin{cases} 1 & \mbox{if } m=\sum_{i=1}^jG_i\text{ for some }j\geq 1; \\ 0 & \mbox{otherwise.} \end{cases}\] Note that $Q_1,Q_2,\ldots$ is a sequence of independent Bernoulli random variables, each with expected value $1-p$. Let $\mu=k+(1-p)\xi$ and $\delta=(1-p)\xi/\mu$ so that $k=(1-\delta)\mu$. Then $\sum_{1\leq m\leq\mu/(1-p)}Q_m$ is a binomial random variable with parameters $\left\lfloor\mu/(1-p)\right\rfloor$ and $1-p$, so we can use a standard Chernoff bound to find that 
\begin{align}\label{eq:Chernoff1}
\nonumber \mathbb P\left(G_1+\cdots+G_{k}\geq k/(1-p)+\xi\right) &= \mathbb P\left(\sum_{1\leq m\leq k/(1-p)+\xi}Q_m\leq k\right) \\ 
\nonumber &=\mathbb P\left(\sum_{1\leq m\leq \mu/(1-p)}Q_m\leq(1-\delta)\mu\right) \\ 
\nonumber &\leq\exp\left(-\mu\delta^2/2\right) \\ &=\exp\left(-\frac{1-p}{2}\frac{\xi^2}{k/(1-p)+\xi}\right).
\end{align} 
Now let $\mu'=k-(1-p)\xi$ and $\delta'=(1-p)\xi/\mu'$ so that $k=(1+\delta')\mu'$. Then $\sum_{1\leq m\leq\mu'/(1-p)}Q_m$ is a binomial random variable with parameters $\left\lfloor\mu'/(1-p)\right\rfloor$ and $1-p$, so we can again use a standard Chernoff bound to find that 
\begin{align}\label{eq:Chernoff2}
\nonumber \mathbb P\left(G_1+\cdots+G_{k}\leq k/(1-p)-\xi\right) &= \mathbb P\left(\sum_{1\leq m\leq k/(1-p)-\xi}Q_m\geq k\right) \\ \nonumber 
&=\mathbb P\left(\sum_{1\leq m\leq \mu'/(1-p)}Q_m\geq(1+\delta')\mu'\right) \\ \nonumber 
&\leq\exp\left(-\mu'(\delta')^2/(2+\delta')\right) \\ \nonumber &=\exp\left(-(1-p)\frac{\xi^2}{2k/(1-p)-\xi}\right) \\ &\leq \exp\left(-\frac{1-p}{2}\frac{\xi^2}{k/(1-p)+\xi}\right). 
\end{align} 
The desired result follows from \eqref{eq:Chernoff1} and \eqref{eq:Chernoff2}. 
\end{proof}

For $z\geq 0$, let \[T_\pp(z)=\inf\{t\geq 0:|\dd_\pp(t)-\dd_\pp(0)|\geq z\}.\]

\begin{lemma}\label{lem:veer} 
There are constants $C_1,C_2>0$ such that for all $r\geq 0$ and $z\geq 2$, we have
\[\mathbb P(T_\pp(z)\leq r)<C_1\exp\left(-C_2\frac{z^2}{z+r}\right).\]   
\end{lemma}

\begin{proof}
Let $\tau_1,\tau_2,\ldots$ be the times when $\pp$ turns, and let $\tau_0=0$. Let $K=\left\lfloor \frac{1-p}{2}(z+r)\right\rfloor$. To prove the lemma, we will show that $\mathbb P(\tau_{2K}\leq r)$ and 
$\mathbb P(T_\pp(z)\leq \tau_{2K})$ are both small. 

First, since $\tau_1-\tau_0,\tau_2-\tau_1,\ldots,\tau_{2K}-\tau_{2K-1}$ are independent geometric random variables with parameter $1-p$, we can apply \cref{lem:Chernoff} with $k=2K$ and $\xi=z$ to find that 
\begin{align}\label{eq:bound1}
\nonumber \mathbb P(\tau_{2K}\leq r)&=\mathbb P\left((\tau_1-\tau_0)+\cdots+(\tau_{2K}-\tau_{2K-1})-2K/(1-p)\leq r-2K/(1-p)\right) \\ \nonumber &\leq \mathbb P\left((\tau_1-\tau_0)+\cdots+(\tau_{2K}-\tau_{2K-1})-2K/(1-p)\leq -z\right) \\ \nonumber &\leq \mathbb P\left(|(\tau_1-\tau_0)+\cdots+(\tau_{2K}-\tau_{2K-1})-2K/(1-p)|\geq z\right) \\
\nonumber &\leq 2\exp\left(-\frac{1-p}{2}\frac{z^2}{2K/(1-p)+z}\right) \\ &\leq 2\exp\left(-\frac{1-p}{4}\frac{z^2}{z+r}\right).
\end{align} 

Now let $T=T_\pp(z)$, and suppose $T\leq \tau_{2K}$. If $\bb_\pp(T)$ is to the east of $\bb_\pp({T-1})$, let $\delta=1$; otherwise, let $\delta=-1$. Then $\delta\cdot(\dd_\pp(T)-\dd_\pp(0))\geq z$. Consider the shape \[\widetilde{\mathscr R}=\{\fb(x,y):x,y\in\ZZ,\,x+y\geq T+c\},\] where $c$ is the integer used to define $\mathscr{R}$ in \eqref{eq:RR}. The part of the pipe $\pp$ that grows starting at time $T$ can be seen as a pipe $\widetilde{\pp}$ in $\widetilde{\mathscr{R}}$. Applying \cref{lem:half} to $\widetilde{\pp}$, we find that \[\mathbb P(\delta\cdot(\dd_\pp({\tau_{2K}})-\dd_\pp(T))\geq -1)\geq\frac{1}{2},\] so \[\mathbb P(|\dd_\pp({\tau_{2K}})-\dd_\pp(0)|\geq z-1)\geq\frac{1}{2}.\]  

The preceding paragraph tells us that the conditional probability of the event that \[|\dd_\pp({\tau_{2K}})-\dd_\pp(0)|\geq z-1\] given that $T\leq \tau_{2K}$ is at least $\frac{1}{2}$. Therefore, 
\begin{align*}
\mathbb P(T\leq \tau_{2K})&\leq 2\mathbb P(|\dd_\pp({\tau_{2K}})-\dd_\pp(0)|\geq z-1\text{ and }T\leq \tau_{2K}) \\ &\leq 2\mathbb P(|\dd_\pp({\tau_{2K}})-\dd_\pp(0)|\geq z-1). 
\end{align*}
Now,
\[\left|\dd_\pp({\tau_{2K}})-\dd_\pp(0)\right|=|A_1-A_2|\leq|A_1|+|A_2|,\] where 
\[A_1=\sum_{i=1}^{K}\left(\tau_{2i}-\tau_{2i-1}\right)-K/(1-p)\quad\text{and}\quad A_2=\sum_{i=1}^{K}\left(\tau_{2i-1}-\tau_{2i-2}\right)-K/(1-p).\] It follows that 
\begin{align*}
\mathbb P(|\dd_\pp(\tau_{2K})-\dd_\pp(0)|\geq z-1)&\leq \mathbb P(|A_1|\geq(z-1)/2)+\mathbb P(|A_2|\geq(z-1)/2) \\ &=2\mathbb P(|A_1|\geq(z-1)/2),
\end{align*} 
where the second equality follows from the observation that $A_1$ and $A_2$ are i.i.d.\ random variables. Applying \cref{lem:Chernoff} with $k=K$ and $\xi=(z-1)/2$, we find that 
\begin{align}\label{eq:bound2} 
\nonumber \mathbb P(T\leq \tau_{2K})&\leq 4\mathbb P(|A_1|\geq (z-1)/2) \\ \nonumber &\leq 8\exp\left(-\frac{1-p}{2}\frac{(z-1)^2/4}{(z+r)/2+(z-1)/2}\right) \\ \nonumber &= 8\exp\left(-\frac{1-p}{4}\frac{(z-1)^2}{2z+r-1}\right) \\ &< 8\exp\left(-\frac{1-p}{4}\frac{z^2/4}{3(z+r)}\right).
\end{align}

Since $\mathbb P(T\leq r)\leq\mathbb P(\tau_{2K}\leq r)+\mathbb P(T\leq \tau_{2K})$, the desired result follows from \eqref{eq:bound1} and \eqref{eq:bound2}.  
\end{proof}

In the remainder of the article, we always let $C_1$ and $C_2$ be the constants from \cref{lem:veer}. 

\section{The Distance Between Two Pipes}\label{sec:distance} 

Continuing to work with a random pipe dream in the shape $\mathscr{R}=\{\fb(x,y):x,y\in\ZZ,\,x+y\geq 0\}$ as in the preceding section, we now simultaneously consider two pipes $\pp$ and $\mathfrak p'$. As before, we imagine that each pipe grows in the north and east directions, moving into a new box at each positive integer time step. Let $\bb_{\qq}(t)$ denote the box containing a pipe $\qq$ at time $t$, and let $\dd_{\qq}(t)=\dd(\bb_{\qq}(t))$. Let us also write $\hh_{\pp,\pp'}(t)=\dd_{\pp'}(t)-\dd_{\pp}(t)$. We will assume that $\hh_{\pp,\pp'}(0)$ is fixed and is nonnegative. Say $\pp$ and $\pp'$ \dfn{kiss} at time $t$ if $\bb_\pp(t)=\bb_{\pp'}(t)$ (equivalently, $\hh_{\pp,\pp'}(t)=0$). We want to understand the \dfn{first kiss} of $\mathfrak p$ and $\mathfrak p'$, which we define to be 
\[\TT_{\heartsuit}=\TT_{\heartsuit}(\mathfrak p,\mathfrak p')=\inf\{t\geq 0:\bb_\pp(t)=\bb_{\pp'}(t)\}.\] 

Letting $\N$ stand for ``north'' and $\E$ stand for ``east,'' let $\ff_t\in\Omega\coloneq\{\N\N,\N\E,\E\N,\E\E\}$ be the string of length $2$ whose first and second letters encode the directions that $\pp$ and $\pp'$, respectively, face at time $t$. For example, if at time $t$ the pipe $\pp$ faces east and the pipe $\pp'$ faces north, then $\ff_t=\E\N$. 

Let ${\bf Y}=(Y_t)_{t\geq 0}$ be the Markov chain with state space $\Omega$ and transition diagram \[\begin{array}{l}\includegraphics[height=2.914cm]{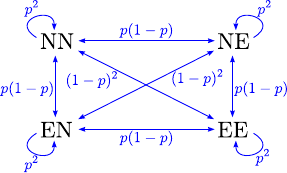}\end{array}.\] 
The transition matrix of ${\bf Y}$ (with rows and columns indexed by states in the order $\N\N,\N\E,\E\N,\E\E$) is \begin{equation}\label{eq:P_matrix}
P=\begin{pmatrix}
    p^2 & p(1-p) & p(1-p) & (1-p)^2 \\
    p(1-p) & p^2 & (1-p)^2 & p(1-p) \\
    p(1-p) & (1-p)^2 & p^2 & p(1-p) \\
    (1-p)^2 & p(1-p) & p(1-p) & p^2 
\end{pmatrix}.
\end{equation} 
If $t$ is a time such that $\bb_\pp(t)\neq\bb_{\pp'}(t)$ and we are given $\ff_t$, then we can easily compute the distribution of $\ff_{t+1}$; more precisely, for $\XX,\XX'\in\Omega$, we have 
\[\mathbb P(\ff_{t+1}=\XX\mid\ff_t=\XX')=\mathbb P(Y_{t+1}=\XX\mid Y_t=\XX').\] 
Moreover, $\hh_{\pp,\pp'}(t+1)-\hh_{\pp,\pp'}(t)=\nu(\ff_{t+1})$, where 
\[\nu(\N\N)=\nu(\E\E)=0,\quad\nu(\N\E)=2,\quad\nu(\E\N)=-2.\] One of the difficulties of working directly with the pipes $\pp$ and $\pp'$ is that we have to worry about the lack of independence in their trajectories when they kiss; this is the advantage of working with the Markov chain ${\bf Y}$ instead. 

For $k\geq 1$, we are interested in the random variable \[\psi_k=\sum_{t=1}^k\nu(Y_t).\] 
Let $\Phi_p\colon\mathbb R\to\mathbb R$ be the cumulative distribution function of a normal random variable with mean $0$ and variance $\frac{2p}{1-p}$. That is, 
\[\Phi_p(x)=\frac{1}{2\sqrt{\pi}}\sqrt{\frac{1-p}{p}}\int_{-\infty}^x \exp\left(-\frac{1-p}{4p}u^2\right)du.\]  

\begin{lemma}\label{lem:Q_constant}
We have 
\[\left|\mathbb P(\psi_k\leq z)-\Phi_p(zk^{-1/2})\right|=O(k^{-1/2})\] for all $k\geq 1$ and $z\in\mathbb R$.  
\end{lemma}

\begin{proof}
The central limit theorem for Markov chains (see \cite{TL}) tells us that the random variable $\frac{1}{\sqrt{k}}\psi_k$ converges in distribution to a normal distribution with some mean $\mu$ and some variance $\sigma^2$. The mean $\mu$ is simply the expected value of $\nu$ with respect to the stationary distribution of ${\bf Y}$. This stationary distribution is uniform, so \[\mu=\frac{1}{4}\nu(\N\N)+\frac{1}{4}\nu(\N\E)+\frac{1}{4}\nu(\E\N)+\frac{1}{4}\nu(\E\E)=0.\] 

Trevezas and Limnios \cite{TL} developed an efficient menthod to compute the variance $\sigma^2$. Following their notation, we let $\Pi$ be the $4\times 4$ matrix each row of which is given by the stationary distribution of ${\bf Y}$; in our setting, $\Pi$ is the $4\times 4$ matrix whose entries are all equal to $\frac{1}{4}$. Let $\Pi_{dg}$ be the matrix obtained from $\Pi$ by changing every off-diagonal entry to $0$; then $\Pi_{dg}=\frac{1}{4}I_4$, where $I_4$ is the $4\times 4$ identity matrix. Let $Z=(I_4-P+\Pi)^{-1}$, where $P$ is as in \eqref{eq:P_matrix}. Let $\widetilde{\boldsymbol{\nu}}=(\nu(\N\N),\nu(\N\E),\nu(\E\N),\nu(\E\E))=(0,-2,2,0)$, and let $\widetilde{\boldsymbol{\nu}}^\top$ be the transpose of $\widetilde{\boldsymbol{\nu}}$. According to \cite[Proposition~2]{TL}, we have 
\[\sigma^2=\widetilde{\boldsymbol{\nu}}(2Z-I_4)\widetilde{\boldsymbol{\nu}}^\top;\] from this, we easily compute that 
\[\sigma^2=\frac{2p}{1-p}.\] 

We have shown that $\mathbb P(\psi_k\leq z)-\Phi_p(zk^{-1/2})\to 0$ as $k\to\infty$. The explicit quantitative bound in the statement of the lemma is a direct consequence of \cite[Theorem~1]{Bolthausen}, which is a version of the Berry--Esseen theorem for Markov chains. (The hypotheses of \cite[Theorem~1]{Bolthausen} include four mild conditions on the Markov chain ${\bf Y}$ and the functional $\nu$, but each one is routine to check.) 
\end{proof} 

For each nonnegative real number $z$, let $\Psi_z=\inf\{k\geq 0:\psi_k\geq z\}$. 

\begin{lemma}\label{lem:Psi}
For all $z\geq 0$, we have 
\[\mathbb P(\Psi_z\leq k)=2-2\Phi_p(zk^{-1/2})+O(k^{-1/6}).\]  
\end{lemma}

\begin{proof}
Let $k'=k+\left\lfloor k^{1/3}\right\rfloor$.
Consider a positive integer $k''\leq k$, and note that $\psi_{k'}-\psi_{k''}$ has the same distribution as $\psi_{k'-k''}$. If $\Psi_z=k''$, then $|\psi_{k''}-z|\leq 2$, so it follows from \cref{lem:Q_constant} that $\mathbb P(\psi_{k'}\geq z)=\frac{1}{2}+O((k'-k'')^{-1/2})=\frac{1}{2}+O(k^{-1/6})$. This shows that \[\mathbb P(\psi_{k'}\geq z\mid\Psi_z\leq k)=\frac{1}{2}+O(k^{-1/6}),\] so 
\begin{align}\label{eq:hello}
\nonumber \mathbb P(\Psi_z\leq k)&=(2+O(k^{-1/6}))\mathbb P(\psi_{k'}\geq z\text{ and }\Psi_z\leq k) \\ &=(2+O(k^{-1/6}))\left[\mathbb P(\psi_{k'}\geq z)-\mathbb P(\psi_{k'}\geq z\text{ and }\Psi_z>k)\right]. 
\end{align} 

If $\psi_{k'}\geq z$ and $\Psi_z>k$, then we must have \[z>\psi_k=\psi_{k'}-\sum_{k<t\leq k'}\nu(Y_t)\geq \psi_{k'}-2k^{1/3}\geq z-2k^{1/3}.\] Let $\Phi_p'$ denote the derivative of $\Phi_p$. Appealing to \cref{lem:Q_constant}, we find that 
\begin{align*}
\mathbb P(\psi_{k'}\geq z\text{ and }\Psi_z>k)&\leq\mathbb P(z-2k^{1/3}\leq\psi_k<z) \\ 
&=\mathbb P(\psi_k<z)-\mathbb P(\psi_k<z-2k^{1/3}) \\ &=\Phi_p(zk^{-1/2})-\Phi_p(zk^{-1/2}-2k^{-1/6})+O(k^{-1/2}) \\ &=O(\Phi_p'(zk^{-1/2})k^{-1/6})+O(k^{-1/2}) \\&=O(k^{-1/6}),
\end{align*} 
where we have used the fact that $\Phi_p$ is Lipschitz. By \eqref{eq:hello} and \cref{lem:Psi}, we have 
\begin{align*}
\mathbb P(\Psi_z\leq k)&=(2+O(k^{-1/6}))\mathbb P(\psi_{k'}\geq z)+O(k^{-1/6}) \\ &=(2+O(k^{-1/6}))(1-\Phi_p(z(k')^{-1/2})+O((k')^{-1/2}))+O(k^{-1/6}) \\ &=2-2\Phi_p(zk^{-1/2})+O(k^{-1/6}). \qedhere
\end{align*} 
\end{proof}

We now return to considering the pipes $\pp$ and $\pp'$ and their first kiss.  

\begin{lemma}\label{lem:Theart}
We have \[\mathbb P(\TT_\heartsuit(\pp,\pp')\leq k)=2-2\Phi_p(\hh_{\pp,\pp'}(0) k^{-1/2})+O(k^{-1/6}).\] 
\end{lemma} 

\begin{proof}
This follows immediately from \cref{lem:Psi} and the observation that $\TT_\heartsuit(\pp,\pp')$ has the same distribution as $\Psi_{\hh_{\pp,\pp'}(0)}$. 
\end{proof}

\section{Serrated Shapes}\label{sec:serrated} 

Let us say that a shape $\bT$ of rank $n-1$ is \dfn{serrated} if there is an integer $c$ such that 
\[\{\fb(x,y):x,y\in\ZZ,\,x+y=c,\,1\leq x-y\leq n-1\}\subseteq\bT\] and $\bT$ does not contain any box $\fb(x,y)$ with $x+y<c$ (note that this definition depends on $n$). The name comes from the appearance of the southwest boundary of the shape (see \cref{fig:serrate}, in which the boxes shaded in periwinkle form a serrated shape). For each $n\geq 2$, fix an order-convex serrated shape $\bT^{(n)}$ of rank $n-1$, and let $\NN_n=\max_{i\in[n-1]}\occ_i(\bT^{(n)})$. Assume that \begin{equation}\label{eq:MM}
n^{2\varepsilon}<\NN_n=o\left(n^2/\log^{1+\varepsilon} n\right).
\end{equation} This implies that $\sqrt{\NN_n\log^{1+\varepsilon}\NN_n}=o(n)$. 

Consider a random pipe dream $\PD_p(\bT^{(n)})$. As before, let $\bb_{\qq}(t)$ be the box containing a pipe $\qq$ at time $t$, and let $\dd_{\qq}(t)=\dd(\bb_{\qq}(t))$. Let $R_{\qq}$ be the largest integer $t$ such that $\bb_{\qq}(t)$ exists; that is, $R_{\qq}+1$ is the number of boxes that intersect $\qq$. Because $\bT^{(n)}$ is order-convex and serrated, it is straightforward to check that  
\begin{equation}\label{eq:RO}
R_{\qq}\leq 2\NN_n.
\end{equation} For distinct pipes $\qq,\qq'$, let $\hh_{\qq,\qq'}(t)=\dd_{\qq'}(t)-\dd_{\qq}(t)$. As before, we say that $\qq$ and $\qq'$ \emph{kiss} at time $t$ if $\bb_{\qq}(t)=\bb_{\qq'}(t)$. Let $\TT_{\heartsuit}(\qq,\qq')=\inf\{t\geq 0:\bb_{\qq}(t)=\bb_{\qq'}(t)\}$ be the first kiss of $\qq$ and $\qq'$. 

Throughout this section, fix pipes $\pp,\pp'$ in $\PD_p(\bT^{(n)})$. Say $\pp$ and $\pp'$ are the $j$-th and $j'$-th pipes, respectively, in $\PD_p(\bT^{(n)})$ (recall that we number pipes from northwest to southeast according to their starting points). 
Let $m=\occ_{\dd_\pp(0)}(\bT^{(n)})$. Let us assume that $m>n^{2\varepsilon}$ and that 
\begin{equation}\label{eq:10}
3\textstyle{\sqrt{\NN_n\log^{1+\varepsilon} \NN_n}}<\dd_\pp(0)<n-3\textstyle{\sqrt{\NN_n\log^{1+\varepsilon} \NN_n}}.
\end{equation}  
If $(\pp,\pp')$ is an inversion, then $\TT_\heartsuit(\pp,\pp')\leq R_\pp$. 

Throughout this section, we tacitly make use of the obvious fact that the trajectory of a pipe in $\bT^{(n)}$ (before it hits the northeast boundary) has the same distribution as it would have if $\bT^{(n)}$ were extended to the infinite shape $\mathscr{R}$ from \cref{sec:concentration}. That is, the trajectory of the pipe is not affected by the northeast boundary of $\bT^{(n)}$ until the pipe actually reaches the northeast boundary. This allows us to apply results that we established for infinite shapes in \cref{sec:concentration,sec:distance}, so long as we can ignore the cases where pipes touch the northwest or southeast boundary of $\bT^{(n)}$. As in \cref{lem:veer}, let 
\[T_\pp(z)=\inf\{t\geq 0:|\dd_\pp(t)-\dd\pp(0)|\geq z\}.\]

\begin{lemma}\label{lem:touch_boundary}
The probability that $(\pp,\pp')$ is an inversion and either $\pp$ or $\pp'$ touches the northwest or southeast boundary of $\bT^{(n)}$ is $O(n^{-100})$. 
\end{lemma} 

\begin{proof}
If $\pp$ or $\pp'$ touches the northwest or southeast boundary of $\bT^{(n)}$, then by \eqref{eq:RO} and \eqref{eq:10}, there must be times $0\leq t<t'\leq 2\NN_n$ and a pipe $\qq\in\{\pp,\pp'\}$ such that $|\dd_{\qq}(t')-\dd_\qq(t)|>\sqrt{\NN_n\log^{1+\varepsilon}\NN_n}$ and such that $\qq$ does not touch the northwest or southeast boundary of $\bT^{(n)}$ between times $t$ and $t'$ (including at times $t$ and $t'$). For any fixed $t$ and $t'$, we know by \cref{lem:veer} (applied to the part of $\qq$ that starts at time $t$) and the fact that $t'-t\leq R_{\qq}\leq 2\NN_n$ that the probability of this occurring is at most \[C_1\exp\left(-C_2\frac{\NN_n\log^{1+\varepsilon} \NN_n}{\sqrt{\NN_n\log^{1+\varepsilon}\NN_n}+(t'-t)}\right)=O(\NN_n^{-100/\varepsilon})=O(n^{-200}).\] The desired result follows by taking a union bound over all possible choices of $t,t',\qq$. 
\end{proof}

\begin{lemma}\label{lem:pleasant1}
The probability that $(\pp,\pp')$ is an inversion and \[\max_{0\leq t\leq R_\pp}|\dd_\pp(t)-\dd_\pp(0)|>\textstyle{\sqrt{m\log^{1+\varepsilon} m}}\quad\text{or}\quad\max_{0\leq t\leq R_{\pp'}}|\dd_{\pp'}(t)-\dd_{\pp'}(0)|>\textstyle{\sqrt{m\log^{1+\varepsilon} m}}\] is $O(n^{-100})$. 
\end{lemma}

\begin{proof}
By \cref{lem:touch_boundary}, we may assume $\pp$ and $\pp'$ do not touch the northwest or southeast boundary of $\bT^{(n)}$. We will prove that \[\mathbb P\left(\max_{0\leq t\leq R_\pp}|\dd_\pp(t)-\dd_\pp(0)|>\textstyle{\sqrt{m\log^{1+\varepsilon} m}}\right)=O(n^{-100});\] an analogous argument shows that \[\mathbb P\left(\max_{0\leq t\leq R_{\pp'}}|\dd_{\pp'}(t)-\dd_{\pp'}(0)|>\textstyle{\sqrt{m\log^{1+\varepsilon} m}}\right)=O(n^{-100}).\]

If $\max_{0\leq t\leq R_\pp}|\dd_\pp(t)-\dd_\pp(0)|>\sqrt{m\log^{1+\varepsilon} m}$, then $T_\pp(\sqrt{m\log^{1+\varepsilon} m})\leq R_\pp$; hence, it suffices to show that $\mathbb P(R_\pp>4m)$ and $\mathbb P(T_\pp(\sqrt{m\log^{1+\varepsilon} m})\leq 4m)$ are both $O(n^{-100})$. Because $\bT^{(n)}$ is order-convex and $m=\occ_{\dd_\pp(0)}(\bT^{(n)})$, the pipe $\pp$ either travels at most $m$ steps north or travels at most $m$ steps east. Therefore, if $R_\pp>4m$, then $|\dd_\pp(R_\pp)-\dd_\pp(0)|\geq R_\pp-2m$. For each integer $a>4m$, we know by \cref{lem:veer} that 
\begin{align*}
\mathbb P(R_\pp=a)&\leq \mathbb P(T_\pp(a-2m)\leq a) \\  &\leq C_1\exp\left(-C_2\frac{(a-2m)^2}{2a-2m}\right) \\ &\leq C_1\exp\left(-C_2\frac{(a/2)^2}{2a}\right)\\ &= C_1\exp\left(-C_2\,a/8\right).
\end{align*}
It follows that \[\mathbb P(R_\pp>4m)=\sum_{a>4m}\mathbb P(R_\pp=a)=O(m^{-50/\varepsilon})=O(n^{-100}).\]
On the other hand, \cref{lem:veer} tells us that    
\[
\mathbb P\left(T_\pp\left(\textstyle{\sqrt{m\log^{1+\varepsilon} m}}\right)\leq 4m\right)  \leq C_1\exp\left(-C_2\frac{m\log^{1+\varepsilon} m}{\sqrt{m\log^{1+\varepsilon} m}+4m}\right) =O(m^{-50/\varepsilon})=O(n^{-100}), 
\] as desired. 
\end{proof}

Note that if $j'-j>2\sqrt{m\log^{1+\varepsilon} m}$, then it follows from \cref{lem:pleasant1} that the probability that $(\pp,\pp')$ is an inversion is $O(n^{-100})$.

\begin{lemma}\label{lem:Phip}
Let $\ell$ and $\ell'$ be quantities that are both $O(\sqrt{m\log^{1+\varepsilon} m})$. Then \[\Phi_p\left(\ell(2m+\ell')^{-1/2}\right)=\Phi_p\left(\ell(2m)^{-1/2}\right)+O\left((\log^{1+\varepsilon} m)m^{-1/2}\right).\] 
\end{lemma}
\begin{proof}
Let $\Phi_p'$ denote the derivative of $\Phi_p$. We compute that 
\begin{align*}
\Phi_p\left(\ell\left(2m+\ell'\right)^{-1/2}\right)&=\Phi_p\left(\ell (2m)^{-1/2}\left(1+O(\ell'/m)\right)\right)
\\ &=\Phi_p\left(\ell (2m)^{-1/2}+O\left((\log^{1+\varepsilon}m)m^{-1/2}\right)\right) \\ &=\Phi_p\left(\ell (2m)^{-1/2}\right)+O\left(\Phi_p'(\ell (2m)^{-1/2})(\log^{1+\varepsilon} m)m^{-1/2}\right) \\ &=\Phi_p\left(\ell(2m)^{-1/2}\right)+O\left((\log^{1+\varepsilon} m)m^{-1/2}\right),
\end{align*} 
where the last equality follows from the fact that $\Phi_p$ is Lipschitz. 
\end{proof} 

\begin{lemma}\label{lem:narrow_first_kiss}
The probability that $(\pp,\pp')$ is an inversion and $\TT_\heartsuit(\pp,\pp')>2m-6\textstyle{\sqrt{m\log^{1+\varepsilon} m}}$ is $O(m^{-1/6})$. 
\end{lemma}

\begin{proof}
Let $m'=\occ_{\dd_{\pp'}(0)}(\bT^{(n)})$. By \cref{lem:touch_boundary,lem:pleasant1}, we can assume that $\pp$ and $\pp'$ do not touch the northwest or southeast boundary of $\bT^{(n)}$ and that 
\[\max_{0\leq t\leq R_\pp}|\dd_\pp(t)-\dd_\pp(0)|\leq\textstyle{\sqrt{m\log^{1+\varepsilon} m}}\quad\text{and}\quad\max_{0\leq t\leq R_{\pp'}}|\dd_{\pp'}(t)-\dd_{\pp'}(0)|\leq\textstyle{\sqrt{m\log^{1+\varepsilon} m}}.\] 
Because $\bT^{(n)}$ is order-convex and $m=\occ_{\dd_\pp(0)}(\bT^{(n)})$, the pipe $\pp$ either travels at most $m$ steps north or travels at most $m$ steps east.
It follows that $|R_\pp-2m|\leq \sqrt{m\log^{1+\varepsilon} m}$. Similarly, we have ${|R_{\pp'}-2m'|\leq\sqrt{m'\log^{1+\varepsilon} m'}}$. Let $\ell=\hh_{\pp,\pp'}(0)$, and note that $|m-m'|\leq\ell$ (since $\bT^{(n)}$ is order-convex). By \cref{lem:pleasant1}, we may assume that $\ell\leq 2\sqrt{m\log^{1+\varepsilon} m}$. Thus, \[|R_{\pp'}-2m|\leq |R_{\pp'}-2m'|+|2m-2m'|\leq 5\textstyle{\sqrt{m\log^{1+\varepsilon} m}}.\] 

If $(\pp,\pp')$ is an inversion and ${\TT_\heartsuit(\pp,\pp')>2m-6\sqrt{m\log^{1+\varepsilon} m}}$, then (since $\pp$ and $\pp'$ must cross at or before time $R_\pp$) we must have 
\begin{equation}\label{eq:event_narrow1}2m-6\textstyle{\sqrt{m\log^{1+\varepsilon} m}}<\TT_\heartsuit(\pp,\pp')<2m+\textstyle{\sqrt{m\log^{1+\varepsilon} m}}.
\end{equation} 
\cref{lem:Theart} tells us that 
\[\mathbb P\left(\TT_\heartsuit(\pp,\pp')<2m+\textstyle{\sqrt{m\log^{1+\varepsilon} m}}\right)\leq 2-2\Phi_p\left(\ell\left(2m+\textstyle{\sqrt{m\log^{1+\varepsilon} m}}\right)^{-1/2}\right)+O(m^{-1/6})\] (the inequality comes from the fact that $\pp$ or $\pp'$ could terminate before time $2m+\sqrt{m\log^{1+\varepsilon} m}$). By \cref{lem:Phip}, we have 
\[\mathbb P\left(\TT_\heartsuit(\pp,\pp')<2m+\textstyle{\sqrt{m\log^{1+\varepsilon} m}}\right)\leq 2-2\Phi_p(\ell (2m)^{-1/2})+O(m^{-1/6}).\]
By a similar computation, 
\[\mathbb P\left(\TT_\heartsuit(\pp,\pp')\leq 2m-6\textstyle{\sqrt{m\log^{1+\varepsilon} m}}\right)=2-2\Phi_p(\ell (2m)^{-1/2})+O(m^{-1/6})\] (here we have an equality since $R_\pp$ and $R_{\pp}$ are both at least $2m-5\sqrt{m\log^{1+\varepsilon} m}$). This implies that the probability of \eqref{eq:event_narrow1} occurring is $O(m^{-1/6})$. 
\end{proof} 

\begin{lemma}\label{lem:jj'}
If $0<j'-j\leq 2\sqrt{m\log^{1+\varepsilon} m}$, then the probability that $(\pp,\pp')$ is an inversion is 
\[(1+o(1))\left(1-\Phi_p\left(\frac{j'-j}{\sqrt{2m}}\right)\right)+O(m^{-1/6}).\] 
\end{lemma}
\begin{proof}
By \cref{lem:touch_boundary,lem:pleasant1}, we can assume that $\pp$ and $\pp'$ do not touch the northwest or southeast boundary of $\bT^{(n)}$ and that 
\[\max_{0\leq t\leq R_\pp}|\dd_\pp(t)-\dd_\pp(0)|\leq\textstyle{\sqrt{m\log^{1+\varepsilon} m}}\quad\text{and}\quad\max_{0\leq t\leq R_{\pp'}}|\dd_{\pp'}(t)-\dd_{\pp'}(0)|\leq\textstyle{\sqrt{m\log^{1+\varepsilon} m}}.\] 
Because $\bT^{(n)}$ is order-convex and $m=\occ_{\dd_\pp(0)}(\bT^{(n)})$, the pipe $\pp$ either travels at most $m$ steps north or travels at most $m$ steps east.
It follows that $|R_\pp-2m|\leq \sqrt{m\log^{1+\varepsilon} m}$. 

Suppose $\TT_\heartsuit(\pp,\pp')\leq 2m-6\sqrt{m\log^{1+\varepsilon} m}$. Let $\mathfrak K$ be the number of times that $\pp$ and $\pp'$ kiss, and let $\mathfrak C$ be the number of times that they cross. Then $\mathfrak C$ is a sum of $\mathfrak K$ independent Bernoulli random variables, each of which takes the value $1$ with probability $p$ and takes the value $0$ with probability $1-p$. Since $R_\pp\geq 2m-\sqrt{m\log^{1+\varepsilon} m}$, we have $R_\pp-\TT_\heartsuit(\pp,\pp')\geq 5\sqrt{m\log^{1+\varepsilon} m}>n^\varepsilon$ (hence, ${R_\pp-\TT_\heartsuit(\pp,\pp')\to\infty}$ as $n\to\infty$), so $\mathfrak K\to\infty$ as $n\to\infty$. This implies that $\mathbb P(\mathfrak C\text{ is odd})=\frac{1}{2}+o(1)$. Consequently, the probability that $(\pp,\pp')$ is an inversion is $\frac{1}{2}+o(1)$. 

The preceding paragraph and \cref{lem:Theart,lem:Phip} tell us that the probability that $\pp$ and $\pp'$ form an inversion is \begin{align*}
&\,\,\,\,\,\,\,\left(\frac{1}{2}+o(1)\right)\mathbb P\left(\TT_\heartsuit(\pp,\pp')\leq 2m-6\textstyle{\sqrt{m\log^{1+\varepsilon} m}}\right)+O(m^{-1/6}) \\ &=\left(\frac{1}{2}+o(1)\right)\left(2-2\Phi_p\left((j'-j)\left(2m-6\textstyle{\sqrt{m\log^{1+\varepsilon} m}}\right)^{-1/2}\right)\right)+O(m^{-1/6}) \\ &=\left(\frac{1}{2}+o(1)\right)\left(2-2\Phi_p\left(\frac{j'-j}{\sqrt{2m}}\right)+O\left((\log^{1+\varepsilon} m)m^{-1/2}\right)\right)+O(m^{-1/6}) \\ &=(1+o(1))\left(1-\Phi_p\left(\frac{j'-j}{\sqrt{2m}}\right)\right)+O(m^{-1/6}). \qedhere
\end{align*} 
\end{proof}

\begin{lemma}\label{lem:end_serrated}
The expected number of pipes $\qq$ in $\PD_p(\bT^{(n)})$ such that $(\pp,\qq)$ is an inversion is \[\left(\sqrt{2/\pi}+o(1)\right)\sqrt{\frac{p}{1-p}}\sqrt{m}.\]
\end{lemma}
\begin{proof}
By \cref{lem:pleasant1,lem:jj'}, the expected number of inversions of the form $(\pp,\qq)$ is 
\begin{align*}
&\,\,\,\,\,\,\,\sum_{i=1}^{2\sqrt{m}\log m}\left((1+o(1))\left(1-\Phi_p\left(\frac{i}{\sqrt{2m}}\right)\right)+O(m^{-1/6})\right)+O(n^{-99}) \\ &=(1+o(1))\sqrt{2m}\int_0^\infty\left(1-\Phi_p(x)\right)dx \\ &=(1+o(1))\sqrt{2m}\int_0^\infty\int_x^\infty\frac{1}{2\sqrt{\pi}}\sqrt{\frac{1-p}{p}}\exp\left(-\frac{1-p}{4p}u^2\right)du\,dx \\ &=(1+o(1))\sqrt{2m}\int_0^\infty\int_0^u\frac{1}{2\sqrt{\pi}}\sqrt{\frac{1-p}{p}}\exp\left(-\frac{1-p}{4p}u^2\right)dx\,du \\ &=\left(\sqrt{2/\pi}+o(1)\right)\sqrt{\frac{p}{1-p}}\sqrt{m}\int_0^\infty\exp\left(-\frac{1-p}{4p}u^2\right)\frac{1-p}{2p}u\,du. 
\end{align*}
Setting $y=\frac{1-p}{4p}u^2$, we obtain that 
\[\int_0^\infty\exp\left(-\frac{1-p}{4p}u^2\right)\frac{1-p}{2p}u\,du=\int_0^\infty\exp(-y)\,dy=1,
\] which implies the desired result. 
\end{proof}

\section{General Shapes}\label{sec:finish_proof}
This section finishes the proof of \cref{thm:main}. For each integer $n\geq 2$, let $\bS^{(n)}$ be a finite order-convex shape of rank $n-1$, and let $\MM_n=\max_{i\in[n-1]}\occ_i(\bS^{(n)})$. Assume that 
\begin{equation}\label{eq:MM_Bounds}\mathfrak{M}_n=o\left(n^2/\log^{1+\varepsilon} n\right) 
\end{equation} and that 
\begin{equation}\label{eq:few_short_pipes}|\{i\in[n-1]:\occ_i(\bS^{(n)})\leq n^{2\varepsilon}\}|=o(n^{1-\varepsilon}).
\end{equation}  
Let 
\[\LL_n=\sqrt{2/\pi}\sqrt{\frac{p}{1-p}}\sum_{i=1}^{n-1}\sqrt{\occ_i(\bS^{(n)})}.\]

\begin{lemma}\label{lem:LL}
We have $\LL_n\gg n^{1+\varepsilon}$ and $\LL_n\gg \MM_n^{3/2}$. 
\end{lemma}

\begin{proof}
The first inequality follows from \eqref{eq:few_short_pipes}, which implies that 
\[\LL_n\gg\sum_{i=1}^{n-1}\sqrt{\occ_i(\bS^{(n)})}\geq (n-o(n))n^{\varepsilon}.\] 

Now, let $i_0\in[n-1]$ be such that $\occ_{i_0}(\bS^{(n)})=\MM_n$. Let us assume $i_0\leq n/2$; a similar argument handles the case in which $i_0>n/2$. If $n$ is sufficiently large, then we know by \eqref{eq:MM_Bounds} that $i_0+\MM_n\leq n$. Because $\bS^{(n)}$ is order-convex, we have $\occ_{i_0+\MM_n-j}(\bS^{(n)})\geq j$ for all $1\leq j\leq \MM_n$. Therefore, 
\begin{align*}\LL_n&\gg\sum_{i=1}^{n-1}\sqrt{\occ_i(\bS^{(n)})} \geq\sum_{j=1}^{\MM_n}\sqrt{\occ_{i_0+\MM_n-j}(\bS^{(n)})} \geq\sum_{j=1}^{\MM_n}\sqrt{j} \gg\MM_n^{3/2}. \qedhere
\end{align*} 
\end{proof}

Consider a random pipe dream $\PD_p(\bS^{(n)})$, and let $v^{(n)}$ be the random permutation that it represents. The inversions of $\PD_p(\bS^{(n)})$ are in bijection with the inversions of $v^{(n)}$, so our goal is to show that the expected number of inversions of $\PD_p(\bS^{(n)})$ is $(1+o(1))\LL_n$.   

Let $\pp_i$ denote the $i$-th pipe in $\PD_p(\bS^{(n)})$, and let $c_i=\alpha_i+\beta_i$ and $d_i=\alpha_i-\beta_i$, where $\fb(\alpha_i,\beta_i)$ is the unique box in $\bS^{(n)}$ containing the starting point of $\pp_i$. Note that $|d_i-i|\leq 1$. We will sometimes refer to $d_i$ and $i$ interchangeably; even though this is not technically correct, it will not affect the relevant asymptotics (in the same way that omitting floor and ceiling symbols does not affect the asymptotics). Let $R_{\qq}+1$ be the number of boxes that intersect a pipe $\qq$ in $\PD_p(\bS^{(n)})$.  

Let 
\[{\bf L}=\left\{i\in[n-1]:i\leq 5\textstyle{\sqrt{\MM_n\log^{1+\varepsilon}\MM_n}}\right\}\quad\text{and}\quad{\bf R}=\left\{i\in[n-1]:i\geq n-5\textstyle{\sqrt{\MM_n\log^{1+\varepsilon} \MM_n}}\right\},\] and let ${\bf M}=[n-1]\setminus({\bf L}\cup{\bf R})$. For $U,U'\subseteq[n-1]$, let $\mathcal E(U,U')$ be the expected number of inversions of $\PD_p(\bS^{(n)})$ of the form $(\pp_i,\pp_j)$ with $i\in U$ and $j\in U'$. Note that 
\begin{equation}\label{eq:LMR}
\mathcal E([n-1],[n-1])=\mathcal E({\bf M},{\bf M})+\mathcal E({\bf M},{\bf R})+\mathcal E({\bf L},{\bf M})+\mathcal E({\bf L},{\bf R})+\mathcal E({\bf L},{\bf L})+\mathcal E({\bf R},{\bf R}).
\end{equation}

\begin{lemma}\label{lem:Final1}
If $j\in{\bf M}$ is such that $\occ_j(\bS^{(n)})>n^{2\varepsilon}$, then  
$\mathcal E(\{j\},{\bf R})=O(n^{-99})$. 
\end{lemma}
\begin{proof}
Let $m=\occ_j(\bS^{(n)})$. The same argument used to prove \cref{lem:touch_boundary} shows that the probability that $\pp_j$ touches the northwest or southeast boundary of $\bS^{(n)}$ is $O(n^{-100})$. Thus, we may assume $\pp_j$ does not touch the northwest or southeast boundary. If $j<j'\leq n-1$, then the same argument used to prove \cref{lem:pleasant1} shows that the probability that $(\pp_j,\pp_{j'})$ is an inversion and \[\max_{0\leq t\leq R_{\pp_j}}|\dd_{\pp_j}(t)-\dd_{\pp_j}(0)|>\textstyle{\sqrt{m\log^{1+\varepsilon} m}}\quad\text{or}\quad\max_{0\leq t\leq R_{\pp_{j'}}}|\dd_{\pp_{j'}}(t)-\dd_{\pp_{j'}}(0)|>\textstyle{\sqrt{m\log^{1+\varepsilon} m}}\] is $O(n^{-100})$. In particular, $\mathcal E(\{j\},\{j'\})=O(n^{-100})$ if $j'\in{\bf R}$. This implies the desired result. 
\end{proof} 

\begin{lemma}\label{lem:bulk} 
If $j\in{\bf M}$ is such that $\occ_j(\bS^{(n)})>n^{2\varepsilon}$, then \[\mathcal E(\{j\},{\bf M}\cup{\bf R})=\left(\sqrt{2/\pi}+o(1)\right)\sqrt{\frac{p}{1-p}}\sqrt{\occ_j(\bS^{(n)})}.\]
\end{lemma}
\begin{proof}
Let $m=\occ_j(\bS^{(n)})$. Consider the shapes \[\mathscr U=\{\fb(x,y):x+y=c_j,\,1\leq x-y\leq n-1\}\quad\text{and}\quad\mathscr{U}'=\{\fb(x,y):x+y\geq c_j,\,1\leq x-y\leq n-1\}.\] Let $\widetilde\bT_j$ be the set of boxes in $\mathscr U'$ that lie weakly southwest of the northeast boundary of $\bS^{(n)}$, and let $\bT_j=\widetilde\bT_j\cup\mathscr U$. (See \cref{fig:serrate} for an illustration.) Let us generate random pipe dreams $\PD_p(\bS^{(n)})$ and $\PD_p(\bT_j)$ that are coupled so that each box in $\bS^{(n)}\cap\bT_j$ is filled with the same type of tile in one pipe dream as in the other. In other words, we can imagine generating $\PD_p(\bT_j)$ from $\PD_p(\bS^{(n)})$ by keeping the tiles in the boxes in $\bS^{(n)}\cap\bT_j$ the same and then choosing new tiles to fill the boxes in $\bT_j\setminus\bS^{(n)}$ (with each box filled with a cross tile with probability $p$ and with a bump tile with probability $1-p$).  
Let $\widetilde\pp_i$ denote the $i$-th pipe in this random pipe dream $\PD_p(\bT_j)$. The key observation is that the number of integers $j'$ such that $(\widetilde\pp_j,\widetilde\pp_{j'})$ is an inversion of $\PD_p(\bT_j)$ is equal to $\mathcal E(\{j\},{\bf M}\cup{\bf R})$. Indeed, this is because $\pp_j$ must lie entirely within boxes in $\bS^{(n)}\cap\bT_j$ (so in fact, $\pp_j=\widetilde\pp_j$). But $\bT_j$ is a serrated shape satisfying the hypotheses of the shape $\bT^{(n)}$ from \cref{sec:serrated}; indeed, it follows from the construction of $\bT_j$ that $\occ_j(\bT_j)=m>n^{2\varepsilon}$ and 
\[\max_{i\in[n-1]}\occ_i(\bT_j)=\MM_n+O(n). \] Therefore, the desired result follows from \cref{lem:end_serrated}. 
\end{proof}

\begin{figure}[]
  \begin{center}
  \includegraphics[height=12cm]{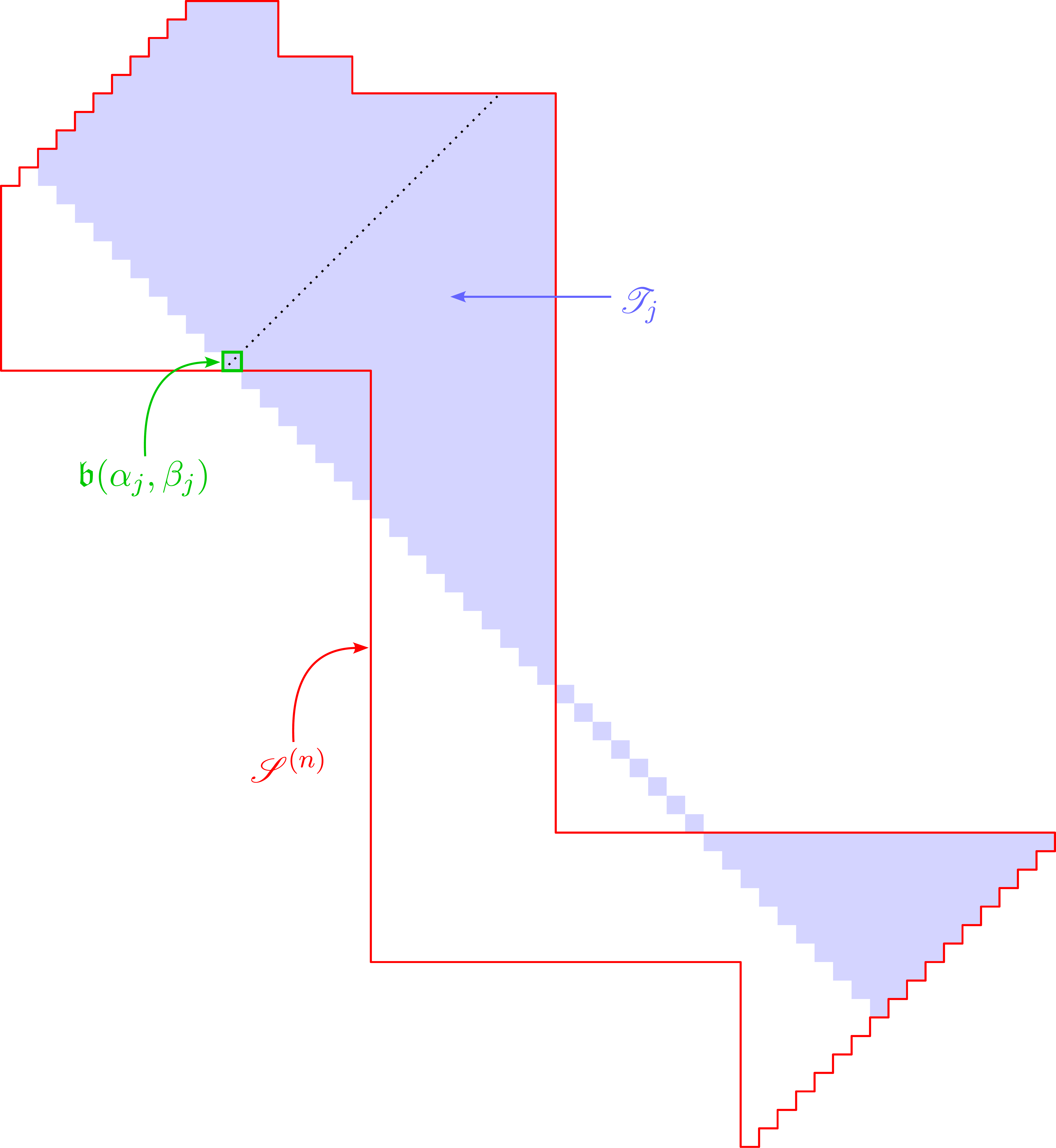}
  \end{center}
\caption{An illustration of the construction of $\bT_j$ from the proof of \cref{lem:Final1}. The shape $\bS^{(n)}$ is outlined in {\color{red}red}. The box $\fb(\alpha_j,\beta_j)$ containing the starting point of $\pp_j$ is outlined in {\color{MildGreen}green}. The shape $\bT_j$ is shaded in {\color{LightBlue}periwinkle}. The number $m=\occ_j(\bS^{(n)})$ counts the boxes whose centers lie on the black dotted line. }\label{fig:serrate} 
\end{figure}

\begin{lemma}\label{lem:short_pipes}
The expected number of inversions of $\PD_p(\bS^{(n)})$ of the form $(\pp_k,\pp_{k'})$ such that $\min\{\occ_k(\bS^{(n)}),\occ_{k'}(\bS^{(n)})\}\leq n^{2\varepsilon}$ is $o(\LL_n)$. 
\end{lemma}
\begin{proof}
Suppose $j\in[n-1]$ is such that $\occ_j(\bS^{(n)})\leq n^{2\varepsilon}$. The pipe $\pp_j$ either travels at most $n^{2\varepsilon}$ steps north or travels at most $n^{2\varepsilon}$ steps east. Therefore, if $R_\pp>2n^{2\varepsilon}$, then $|\dd_{\pp_j}(R_{\pp_j})-\dd_{\pp_j}(0)|\geq R_{\pp_j}-2n^{2\varepsilon}$. Hence, for each integer $a>4n^{2\varepsilon}$, it follows from \cref{lem:veer} that 
\begin{align*}
\mathbb P(R_{\pp_j}=a) &\leq C_1\exp\left(-C_2\frac{(a-2n^{2\varepsilon})^2}{2a-2n^{2\varepsilon}}\right) \\ &\leq C_1\exp\left(-C_2\frac{(a/2)^2}{2a}\right)\\ &= C_1\exp\left(-C_2\,a/8\right).
\end{align*}
Consequently, \[\mathbb P(R_{\pp_j}>4n^{2\varepsilon})=\sum_{a>4n^{2\varepsilon}}\mathbb P(R_{\pp_j}=a)=O(n^{-100}).\]
The number of inversions involving $\pp_j$ is at most $R_\pp+1$. Therefore, the expected number of inversions involving $\pp_j$ is at most 
\[(4n^{2\varepsilon}+1)\mathbb P(R_{\pp_j}\leq 4n^{2\varepsilon})+n\,\mathbb P(R_{\pp_j}>4n^{2\varepsilon})=(4+o(1))n^{2\varepsilon}.\] 

The preceding argument implies that the expected number of inversions of $\PD_p(\bS^{(n)})$ of the form $(\pp_k,\pp_{k'})$ such that $\min\{\occ_k(\bS^{(n)}),\occ_{k'}(\bS^{(n)})\}\leq n^{2\varepsilon}$ is at most 
\[(4+o(1))n^{2\varepsilon}\cdot|\{i\in[n-1]:\occ_i(\bS^{(n)})\leq n^{2\varepsilon}\}|.\] Hence, the desired result follows from \eqref{eq:few_short_pipes} and \cref{lem:LL}. 
\end{proof}

It is immediate from \cref{lem:LL,lem:Final1,lem:bulk,lem:short_pipes} that 
\begin{equation}\label{eq:A1}
\mathcal E({\bf M},{\bf M})=(1+o(1))\LL_n\quad\text{and}\quad\mathcal E({\bf M},{\bf R})=o(\LL_n).
\end{equation}  
Let $\mathrm{ref}(\PD_p(\bS^{(n)}))$ denote the pipe dream obtained by reflecting $\PD_p(\bS^{(n)})$ through the line ${\{(x,y)\in\mathbb R^2:x-y=n/2\}}$. Applying the same argument used to show that $\mathcal E({\bf M},{\bf R})=o(\LL_n)$ to the pipe dream $\mathrm{ref}(\PD_p(\bS^{(n)}))$ instead of $\PD_p(\bS^{(n)})$, we deduce that 
\begin{equation}\label{eq:A2}
\mathcal E({\bf L},{\bf M})=o(\LL_n). 
\end{equation}

Because $|{\bf L}|=|{\bf R}|=\left\lfloor 5\sqrt{\MM_n\log^{1+\varepsilon} \MM_n}\right\rfloor$, it follows from \cref{lem:LL} that 
\begin{equation}\label{eq:A3}
\mathcal E({\bf L},{\bf R})\leq 25\MM_n\log^{1+\varepsilon} \MM_n=o(\LL_n).
\end{equation}
By the same argument, we have 
\begin{equation}\label{eq:A4}
\mathcal E({\bf L},{\bf L})=o(\LL_n)\quad\text{and}\quad\mathcal E({\bf R},{\bf R})=o(\LL_n). 
\end{equation}

Combining \eqref{eq:LMR}, \eqref{eq:A1}, \eqref{eq:A2}, \eqref{eq:A3}, and \eqref{eq:A4}, we deduce that $\mathcal E([n-1],[n-1])$, which is the expected number of inversions of $\PD_p(\bS^{(n)})$ is
$(1+o(1))\LL_n$. This completes the proof of \cref{thm:main}. 

\section{Future Directions}\label{sec:conclusion}  

\subsection{Variance}
\cref{thm:main} concerns the expected value of $\inv(v^{(n)})$, where $v^{(n)}$ is the permutation represented by the random subword $\sub_p(\mathsf{w}^{(n)})$. It would be interesting to say something about the variance of $\inv(v^{(n)})$. 

\subsection{Arbitrary Reduced Words}
One very natural setting in which to apply \cref{thm:main} is when $\mathsf{w}^{(n)}$ is a reduced word for the decreasing permutation in $\SSS_n$. Note, however, that the theorem only applies if $\mathsf{w}^{(n)}$ is alternating. It would be interesting to extend \cref{thm:main} to handle the case when $\mathsf{w}^{(n)}$ is an arbitrary reduced word for the decreasing permutation. 

\subsection{Large Shapes} 
For each $n\geq 2$, let $\mathsf{w}^{(n)}$ be an alternating word over $\{s_1,\ldots,s_{n-1}\}$, and let $v^{(n)}$ be the permutation represented by $\sub_p(\ww^{(n)})$. One hypothesis of \cref{thm:main} requires that \[\max_{i\in[n-1]}\occ_i(\mathsf{w}^{(n)})=o(n^2/\log^{1+\varepsilon} n).\] It is possible that the conclusion of the theorem could still hold if this hypothesis is slightly weakened. However, it cannot be weakened too much. Indeed, if $\max_{i\in[n-1]}\occ_i(\mathsf{w}^{(n)})$ grows too quickly, then the conclusion of \cref{thm:main} will necessarily fail on the grounds that a permutation in $\SSS_n$ cannot have more than $\binom{n}{2}$ inversions. It would be fascinating to understand the asymptotics of $\mathbb E(\inv(v^{(n)}))$ in this regime where $\max_{i\in[n-1]}\occ_i(\ww^{(n)})$ grows quickly. 

For a concrete example, fix a constant $\Lambda>\frac{\pi}{8}\cdot\frac{1-p}{p}$. For each $n\geq 2$, let $\varrho_n$ be an integer such that $\varrho_n\geq\Lambda n^2$, and let $v^{(n)}$ be the random permutation generated as in \cref{exam:bipartite2} (so $\mathsf{w}^{(n)}=(\mathsf{x}_n\mathsf{y}_n)^{\varrho_n}$). It would be interesting to understand $\mathbb E(\inv(v^{(n)}))$ since the sequence $(\varrho_n)_{n\geq 2}$ grows so quickly that the asymptotic formula in \cref{exam:bipartite2} cannot hold.

\subsection{Other Permutation Patterns} 
Consider permutations $u\in\SSS_k$ and $w\in\SSS_n$. We call $u$ a \dfn{pattern}. An \dfn{occurrence} of the pattern $u$ in the permutation $w$ is a tuple $(i_1,\ldots,i_k)$ such that ${1\leq i_1<\cdots<i_k\leq n}$ and such that for all $j,j'\in[k]$, we have $w(i_j)<w(i_{j'})$ if and only if $u(j)<u(j')$. Note that $\inv(w)$ is the number of occurrences of the pattern $21$ in $w$. It would be very interesting to understand (asymptotically) the expected number of occurrences of other patterns in the permutation $v^{(n)}$ from \cref{thm:main}. 
In a similar vein, we can ask for the expected length of the longest decreasing subsequence in $v^{(n)}$. See \cite{BDJ,Bassino,Borga, EvenZohar,Kenyon,Romik} for similar questions about pattern occurrences and decreasing (or increasing) subsequences in random permutations with other distributions.    

\subsection{Other Types}
The symmetric group $\SSS_n$ is an example of a Coxeter group, and $\{s_1,\ldots,s_{n-1}\}$ is its set of simple reflections (see \cite{BjornerBrenti}). The model of random subwords that we have considered here has an obvious generalization to arbitrary Coxeter groups. It seems fruitful to consider this generalization, perhaps for infinite Coxeter groups. 

Let $\widetilde\SSS_n$ denote the affine symmetric group, and let $\widetilde s_0,\widetilde s_1,\ldots,\widetilde s_{n-1}$ be the simple reflections of $\widetilde\SSS_n$ listed in clockwise order around the Coxeter graph (which is a cycle). Let $\widetilde{\mathsf{c}}$ be the word $\widetilde s_0\widetilde s_1\cdots\widetilde s_{n-1}$. Both the current article and the recent article \cite{MPPY} by Morales, Panova, Petrov, and Yeliussizov study the element of $\SSS_n$ represented by the (usual) product of a random subword of the staircase reduced word $(s_{n-1})(s_{n-2}s_{n-1})\cdots(s_1s_2\cdots s_{n-1})$. 
However, Morales, Panova, Petrov, and Yeliussizov also studied the permutation represented by the Demazure product of this random subword, which exhibits drastically different behavior (for example, the number of inversions is on the order of $n^2$ instead of $n^{3/2}$). In the context of combinatorial billiards, the current author has recently studied the element of $\widetilde\SSS_n$ represented by the Demazure product of a random subword of $\widetilde{\mathsf{c}}^k$ (for large $k$) \cite{DefantStoned}. There is thus a natural gap in the literature: it seems that no one has investigated the element of $\widetilde\SSS_n$ represented by the (usual) product of a random subword of $\widetilde{\mathsf{c}}^k$. The current author plans to (at least partially) fill this gap in the future. 

\section*{Acknowledgments}
The author was supported by the National Science Foundation under Award No.\ 2201907 and by a Benjamin Peirce Fellowship at Harvard University. He is very grateful to Jonas Iskander, Noah Kravitz, Matthew Nicoletti, Greta Panova, and Alan Yan for helpful conversations.


\begin{thebibliography}{99}

\bibitem{AlonSpencer}
N. Alon and J. H. Spencer. The probabilistic method, fourth edition. Wiley, 2016. 

\bibitem{Angel}
O. Angel, D. Dauvergne, A. Holroyd, and B. Vir\'ag. The
local limit of random sorting networks. \emph{Ann. Inst. H. Poincar\'e Probab. Statist.}, {\bf 55} (2019), 412--440.

\bibitem{Angel2} 
O. Angel, A. Holroyd, D. Romik, and B. Vir\'ag. Random sorting networks. \emph{Adv. Math.}, {\bf 215} (2007), 839--868. 

\bibitem{BDJ}
J. Baik, P. Deift, and K. Johansson. On the distribution of the length of the
longest increasing subsequence of random permutations. \emph{J. Amer.
Math. Soc.}, {\bf 12} (1999), 1119--1178.

\bibitem{Bassino}
F. Bassino, M. Bouvel, M. Drmota, V. F\'eray, L. Gerin, M. Maazoun, and A. Pierrot. Linear-sized independent sets in random cographs and increasing subsequences in separable permutations. \emph{Comb. Theory}, {\bf 2} (2022). 

\bibitem{BergeronBilley}
N. Bergeron and S. Billey. RC-graphs and Schubert polynomials. \emph{Experiment. Math.}, {\bf 2} (1993), 257--269. 

\bibitem{BCCP}
N. Bergeron, N. Cartier, C. Ceballos, and V. Pilaud. Lattices of acyclic pipe dreams. arXiv:2303.11025. 

\bibitem{BCP}
N. Bergeron, C. Ceballos, and V. Pilaud. Hopf dreams and diagonal harmonics. \emph{J. Lond. Math. Soc.}, {\bf 105} (2022), 1546--1600. 

\bibitem{BjornerBrenti}
A. Bj\"orner and F. Brenti. Combinatorics of Coxeter groups. Springer, 2005. 

\bibitem{Bolthausen}
E. Bolthausen. The Berry--Esseen theorem for functionals of discrete Markov chains. \emph{Z. Wahrsch. Verw. Gebiete}, {\bf 54} (1980), 59--73.   

\bibitem{Borga}
J. Borga. The skew Brownian permuton: a new universality class for random constrained permutations. \emph{Proc. Lond. Math. Soc.}, {\bf 126} (2023), 1842--1883. 

\bibitem{Cartier}
N. Cartier. Lattices properties of acyclic alternating pipe dreams. arXiv:2311.10435. 

\bibitem{Dauvergne} 
D. Dauvergne. The Archimedian limit of random sorting networks. \emph{J. Amer. Math. Soc.}, {\bf 35} (2022), 1215--1267. 

\bibitem{DauvergneVirag}
D. Dauvergne and B. Vir\'ag. Circular support in random sorting networks. \emph{Trans. Amer. Math. Soc.}, {\bf 373} (2020), 1529--1553.

\bibitem{DefantStoned}
C. Defant. Random combinatorial billiards and stoned exclusion processes. arXiv:2406.07858. 

\bibitem{EvenZohar}
C. Even-Zohar. Patterns in random permutations. \emph{Combinatorica}, {\bf 40} (2020), 775--804. 

\bibitem{FominKirillov}
S. Fomin and A. N. Kirillov. Grothendieck polynomials and the Yang--Baxter equation. \emph{Proceedings of the 6th International Conference on Formal Power Series and Algebraic Combinatorics. Discrete Math. Theor. Comput. Sci.}, {\bf 24} (1994). 

\bibitem{GGJNP}
C. Gaetz, Y. Gao, P. Jiradilok, G. Nenashev, and A. Postnikov. Repeatable patterns and the maximum multiplicity of a generator in a reduced word. To appear in \emph{Comb. Theory}. arXiv:2204.03033.  

\bibitem{HPW}
Z. Hamaker, O. Pechenik, and A. Weigandt. Gr\"obner geometry of Schubert polynomials through ice. \emph{Adv. Math.}, {\bf 398} (2022). 

\bibitem{Kenyon}
R. Kenyon, D. Kr\'al', C. Radin, and P. Winkler. Permutations with fixed pattern densities. \emph{Random Structures
Algorithms}, {\bf 56} (2020), 220--250. 

\bibitem{KnutsonMiller1}
A. Knutson and E. Miller. Gr\"obner geometry of Schubert polynomials. \emph{Ann. Math.}, {\bf 161} (2005), 1245--1318.

\bibitem{KnutsonMiller2}
A. Knutson and E. Miller. Subword complexes in Coxeter groups. \emph{Adv. Math.}, {\bf 184} (2004), 161--176. 

\bibitem{MPPY}
A. H. Morales, G. Panova, L. Petrov, and D. Yeliussizov. Grothendieck shenanigans: permutons from pipe dreams via integrable probability. arXiv:2407.21653. 

\bibitem{Pilaud1}
V. Pilaud and F. Santos. The brick polytope of a sorting network. \emph{European J. Combin.},
{\bf 33} (2012), 632--662. 

\bibitem{Pilaud2}
V. Pilaud and C. Stump. Brick polytopes of spherical subword complexes and generalized
associahedra. \emph{Adv. Math.}, {\bf 276} (2015), 1--61. 

\bibitem{Romik}
D. Romik. The surprising mathematics of longest increasing subsequences. Cambridge University Press, 2015. 

\bibitem{TL}
S. Trevezas and N. Limnios. Variance estimation in the central limit theorem for Markov chains. \emph{J. Statist. Plann. Inference}, {\bf 139} (2009), 2242--2253. 
 
\end{thebibliography}
\end{document}